\documentclass[a4paper,11pt]{amsart}
\usepackage[english]{babel}
\usepackage{amsmath,amsfonts,amssymb,amsthm,relsize,setspace,nicefrac,yhmath,amscd,eucal}
\usepackage{amsrefs, mathrsfs}
\usepackage{mathtools,tkz-euclide,tikz-3dplot,tkz-tab}
\usepackage{xcolor}
\usepackage{tikz}
\usepackage{enumitem}
\usetikzlibrary{calc}
\usepackage{comment}
\usepackage{anyfontsize}
\usepackage{float}
\usepackage{makecell}
\usepackage{rotating}
\newcolumntype{M}[1]{>{\centering\arraybackslash}m{#1}}
\newcolumntype{L}[1]{>{\raggedright\arraybackslash}p{#1}}
\newcolumntype{C}[1]{>{\centering\arraybackslash}p{#1}}
\newcolumntype{R}[1]{>{\raggedleft\arraybackslash}p{#1}}

\usepackage{tikz,multicol,lipsum}
\usepackage{multirow}
\usepackage[colorlinks, linkcolor=red, citecolor=blue, urlcolor=blue, hypertexnames=true]{hyperref}

\usepackage{hyperref} 
\usepackage{cleveref}

\usepackage[active]{srcltx} 
\usepackage{color,soul} 
\newtheorem{theorem}{Theorem}
\newtheorem{corollary}[theorem]{Corollary}
\newtheorem{lemma}[theorem]{Lemma}

\usepackage[toc,page]{appendix}

\newtheorem{remark}{Remark}

\theoremstyle{definition}
\newtheorem{definition}[theorem]{Definition}

\newcommand{\be}{\begin{equation}}
\newcommand{\bel}[1]{\begin{equation}\label{#1}}
\newcommand{\ee}{\end{equation}}

\newcommand{\barr}{\begin{eqnarray}}
\newcommand{\earr}{\end{eqnarray}}
\newcommand{\bars}{\begin{eqnarray*}}
\newcommand{\ears}{\end{eqnarray*}}

\newtheorem{subn}{\name}

\newcommand{\bsn}[1]{\def\name{#1}\begin{subn}}
\newcommand{\esn}{\end{subn}}

\newtheorem{sub}{\name}[section]

\newcommand{\bs}{\begin{sub}}
\newcommand{\es}{\end{sub}}

\newcommand{\bth}[1]{\def\name{Theorem}
\begin{sub}\label{t:#1}}
\newcommand{\blemma}[1]{\def\name{Lemma}
\begin{sub}\label{l:#1}}
\newcommand{\bcor}[1]{\def\name{Corollary}
\begin{sub}\label{c:#1}}
\newcommand{\bdef}[1]{\def\name{Definition}
\begin{sub}\label{d:#1}}
\newcommand{\bprop}[1]{\def\name{Proposition}
\begin{sub}\label{p:#1}}

\newcommand{\BA}{\begin{array}}
\newcommand{\EA}{\end{array}}
\newcommand{\BAN}{\renewcommand{\arraystretch}{1.2}
\setlength{\arraycolsep}{2pt}\begin{array}}
\newcommand{\BAV}[2]{\renewcommand{\arraystretch}{#1}
	\setlength{\arraycolsep}{#2}\begin{array}}
\newcommand{\BSA}{\begin{subarray}}
	\newcommand{\ESA}{\end{subarray}}

\newcommand{\BAL}{\begin{aligned}}
	\newcommand{\EAL}{\end{aligned}}
\newcommand{\BALG}{\begin{alignat}}
	\newcommand{\EALG}{\end{alignat}}
\newcommand{\BALGN}{\begin{alignat*}}
	\newcommand{\EALGN}{\end{alignat*}}
\def\angb<#1>{\langle #1 \rangle}

\def\R{\mathbb{R}}

\let\.=\cdot
\let\0=\emptyset

\numberwithin{equation}{section}

\theoremstyle{definition}

\let\.=\cdot
\let\0=\emptyset

\newenvironment{formula}[1]{\begin{equation}\label{eq:#1}}
	{\end{equation}\noindent}

\def\Fi#1{\begin{formula}{#1}}
	\def\Ff{\end{formula}\noindent}

\setstcolor{red}

\usepackage{authblk}
\usepackage{changepage}

\title[]{Liouville--Type Results for Infinity Elliptic Equations Involving Gradient and Hardy--H\'enon Nonlinearities}

\makeatletter
\author[$1$]{Tan-Dat Khuu} \let\Author\@author
\author[$2$]{Trung-Hieu Huynh} 
\author[$3,*$]{Hoang-Hung Vo} 

\affil[$1$]{Faculty of Mathematics and Statistics, Ton Duc Thang University, 19 Nguyen Huu Tho, Tan Hung, Ho Chi Minh City, Vietnam.}
\affil[$2$]{Department of Mathematics, Ho Chi Minh University of Education, 280 An Duong Vuong, Cho Quan, Ho Chi Minh City, Vietnam.}
\affil[$3$]{Faculty of Mathematics and Applications, Saigon University, 273 An Duong Vuong, Cho Quan, Ho Chi Minh City, Vietnam}

\begin{document}
	
	\date{\today}
	
	
	\maketitle
	\begin{adjustwidth}{1.5cm}{1.5cm}
		\begin{center}
			
			\let\thefootnote\relax\footnotetext{\textit{$^{1}$ Email address: tandatkhuu2k3@gmail.com}}
			\let\thefootnote\relax\footnotetext{\textit{$^2$Email address: huynhhieu2004@gmail.com}}											
			\let\thefootnote\relax\footnotetext{\textit{$^{3,*}$Corresponding author. Email address: vhhung@sgu.edu.vn}}
			
			\Author
		\end{center}
	\end{adjustwidth}
	
	\begin{abstract}
		In this paper, we investigate Liouville properties for  degenerate elliptic equations involving the infinity Laplace operator with nonlinear lower-order terms, of the form
		\begin{align*}
			\Delta_\infty^\beta u - cH(u,\nabla u) - \lambda f(|x|,u) = 0 \text{ in } \mathbb R^n,
		\end{align*}
		where $\beta\in [0,2], \Delta_\infty^\beta$ denotes the fractional infinity Laplacian, and the nonlinear terms $H$ and $f$ are Hamiltonian and Hardy-Hénon type models, respectively. Our work extends existing Liouville frameworks for the classical and normalized infinity Laplacian by developing a new weighted comparison principle and locally Lipschitz regularity estimates. To obtain Liouville type theorems, we establish growth conditions for bounded nonnegative viscosity solutions when $f$ corresponds to increasing power-type nonlinearities, that is, when $f$ grows as $u^\gamma$. Furthermore, we analyze the case of an exponential nonlinearity $f$, which grows as $e^u$, showing that this problem exhibits a strongly supercritical behavior; under suitable growth assumptions on $u$ at infinity, only partial Liouville-type conclusions can be obtained. Our techniques combine radial analysis, barrier constructions, and refined comparison arguments, providing a unified framework connecting regularity, comparison principle, and Liouville properties.	
	\end{abstract}
	
	\tableofcontents

	\section{\bf Introduction}
	\label{sec: 1}
	
	The infinity Laplace operator is a second-order nonlinear PDE operator that describes the functional behavior of functions minimizing the Lipschitz norm. Formally, it is defined by
	\begin{align*}
		\Delta_{\infty}u = \langle D^2u\,\nabla u, \nabla u \rangle 
		= \sum_{i,j=1}^{n} \frac{\partial u}{\partial x_i}\frac{\partial^2 u}{\partial x_i \partial x_j}\frac{\partial u}{\partial x_j},
	\end{align*}
	where $\nabla u$ and $D^2u$ denote the gradient and the Hessian matrix of $u$, respectively. The corresponding equation
	\begin{align*}
		\Delta_{\infty}u = 0,
	\end{align*}
	is known as the infinity Laplace equation, which governs the class of $\infty$-harmonic functions. The notion of infinity Laplacian originally comes from the pioneering works of Aronsson \cite{A1965,A1966,A1967} in the 1960s to study \textit{absolutely minimizing Lipschitz} functions. More precisely, the infinity Laplacian equation arises as the Euler-Lagrange equation associated with the $L^{\infty}$-variational problem
	\begin{align*}
		\min_{u} \|\nabla u\|_{L^{\infty}(\Omega)},
	\end{align*}
	which describes absolute minimizers for the supremal functional $I(u) = \|\nabla u\|_{L^{\infty}(\Omega)}$ with prescribed boundary data. Notably, Jensen in \cite{J1973} proved the uniqueness of \textit{absolutely minimal Lipschitz extensions} by considering viscosity solutions to the infinity Laplacian equation, which has motivated many researchers to study such equations. Infinity Laplacian also arises in stochastic tug-of-war games \cite{YOSD2009}, in which two players try to move a token in an open set $U$ toward a favorable spot on the boundary $\partial U$ corresponding to a given payoff function $g$ on $\partial U$. It is worth noting that the payoff is not necessarily fixed; specifically, it may change each time the token is moved, which can be controlled by probabilistic methods (see \cite{MT2017}). Additionally, other applications of infinity Laplacian can be found in mass transfer problems \cite{LW1999} and image processing \cite{VJC1998}. 
	
	In recent years, attention has shifted toward more general and hybrid operators combining $\Delta_{\infty}$ with nonlinear gradient and absorption terms, which appear naturally in models of optimal control, stochastic games, and reaction-diffusion phenomena. A particularly interesting class of such operators is given by
	\begin{align*}
		\Delta_\infty^\beta u = \frac{1}{|\nabla u|^\beta}\Delta_{\infty}u= \frac{1}{|\nabla u|^\beta}\sum_{i,j = 1}^{n}\frac{\partial u}{\partial x_i}\frac{\partial^2 u}{\partial x_i \partial x_j}\frac{\partial u}{\partial x_j} \mbox{ for all } \beta \in [0,2].
	\end{align*}
	This operator interpolates between the classical infinity Laplacian for $\beta = 0$ and the normalized infinity Laplacian for $\beta = 2$, providing a useful framework for multi-scale optimization and degenerate diffusion problems. In addition, the combination with Hamiltonian terms and Hardy-Hénon type nonlinearities leads to rich mathematical structures involving nonhomogeneous weights and singular absorption-reaction balances.
	
	In this paper, we aim to study Liouville theorems, addressing the nonexistence or existence of positive solutions to the nonlinear equations
	\begin{align}\label{main}
		\Delta_\infty^\beta u - cH(u,\nabla u) - \lambda f(|x|,u) = 0 \text{ in } \mathbb R^n,
	\end{align}
	where $\beta\in [0,2]$ and $\lambda > 0$. Here $c$ is chosen for each model of $H(u,\nabla u)$ so that $cH(u,\nabla u)\ge0$ for all relevant $H(u,\nabla u)$; in particular, if $H$ is negative, $c$ is also taken to be negative, and the nonlinearities are modeled by 
	\begin{align}\label{Hf}
		H(u,\nabla u) = \left\{
		\begin{array}{ll}
			~~~~~|\nabla u|^m,\\
			-u^q|\nabla u|^m, \\ 
			~~u^q |\nabla u|^m,  
		\end{array}
		\right. \mbox{ and }f(|x|,u) = \left\{
		\begin{array}{ll}
			|x|^\alpha(u^+)^\gamma,\\
			(1+|x|^2)^{-\alpha}u^{\gamma},\\
			|u|^{\gamma} e^u. 
		\end{array}
		\right.
	\end{align}
	where $m,q$ are positive constants and the principal objective is to identify optimal relations among the parameters $\gamma, \alpha$ that ensure nonnegative viscosity solutions of \eqref{main} satisfying suitable conditions are trivial ($u \equiv 0$), thereby establishing a unified Liouville property for this broad class of fractional infinity Laplacian equations.
	
	Our study is inspired by the series of recent developments in the theory of Liouville-type theorems for infinity Laplace equations and their extensions. As is known, Liouville properties play a fundamental role in the qualitative study of nonlinear partial differential equations due to their crucial role in providing  information about natural phenomena \cite{BHN2010,BH2012} and regularity theory \cite{CEF2001,AAE2025,Armstrong2011,L2014}. This line of research has developed rapidly over the years for a variety of different equations involving fully nonlinear elliptic operators. Related publications include works on symmetry and overdetermined problems for the infinity-Laplacian \cite{BK2011,CF20151}, inhomogeneous infinity Laplacian \cite{AH2022,CF20152,LW20081,LW20082}, H\"ormander operators \cite{BMP2025}, and truncated Laplacian \cite{BGI2021}. For the standard $\infty$-Laplace equations, Crandall, Evans and Gariepy \cite{CEF2001} showed that any infinity harmonic supersolutions which are bounded below are necessarily constant. Further well-known results were proved by Araújo, Leitão, and Teixeira in  \cite{ALT2016}, concerning the Liouville properties for the infinity Laplace equations with strong absorptions
	\begin{align}
		\label{theo5.1.1}
		\Delta_\infty u = \lambda (u^+)^\gamma\mbox{ in } \mathbb{R}^n,
	\end{align} 
	where $\lambda > 0$ and $0 \leq \gamma < 3$. More precisely, if $u$ is a viscosity solution to \eqref{theo5.1.1} and satisfies the following growth condition
	\begin{align*}
		\limsup_{|X|\to\infty} \frac{u(X)}{|X|^{\frac{4}{3-\gamma}}} < \sqrt[3-\gamma]{ \lambda \cdot \frac{(3-\gamma)^4}{64(1+\gamma)}}, 
	\end{align*}
	then $u \equiv 0$. Turning to the combination of the infinity Laplacian with gradient, this formulation stems from tug-of-war games (see \cite{YOSD2009}). We mention \cite{BGI2021,Armstrong2011,P2011,AH2022} as works related to this problem. The Liouville theorems for fractional infinity Laplacian were first proposed by Lu and Wang \cite{LW20081,LW20082}. These authors applied the standard viscosity solution techniques, including comparison principles and Perron's methods, to establish both the existence and uniqueness of solutions to
	\begin{align*}
		\Delta_{\infty}^{\beta}u=f(x)\text{ in }\Omega,
	\end{align*}
	where $\Omega$ is a bounded domain, $\beta\in \{0,2\}$ with Dirichlet boundary condition, provided the data $f$ has a constant sign. Recently, Biswas and Vo in \cite{AH2020} established Liouville results for fractional infinity Laplacian equations involving gradient. Their result asserts that any supersolutions $u$ of
	\begin{align*}
		\Delta_{\infty}^{\beta}u=c|\nabla u|^{4-\beta}\text{ in }\mathbb R^n,
	\end{align*}
	with $c\geq 0$, which are bounded below are necessarily constant. 
	
	In addition to the fractional infinity Laplacian, Hardy-Hénon nonlinearities are also a notable aspect in our study. Over the past few decades, there have been several works on Liouville results for Hardy-Hénon type equations. For instance, the works \cite{PS2012,DQ2020,V2017,CY2021} proposed Liouville type results for Hardy-Hénon equations involving Laplace operators. Furthermore, the nonexistence and existence of solutions to higher-order Hardy-Hénon equations have also captured much attention. In the recent work of Chen, Dai, and Qin \cite{CDQ2023}, the authors studied Liouville-type theorems, a priori estimates, and existence results for critical and super-critical order Hardy-Hénon type equations of the form
	\begin{align}\label{1.4}
		(-\Delta)^{m}u=|x|^{\sigma}u^p\text{ in }\mathbb R^n,
		\end{align}
	where $m\in (0,1],\sigma>-2m,$ and $p>1$. In addition, we mention Giga and Ngô \cite{NY2022}, who established Liouville properties for the problem \eqref{1.4} with $m\geq 2, \sigma>-2m$, and $p>1$. However, to the best of our knowledge, there are not many well-known publications concerning this class of equations for the infinity-Laplacian, only \cite{JSNS2024} has treated it with the infinity Laplacian. It motivates us to consider these equations with the fractional infinity-Laplacian, aiming to generalize known results.
	
	Furthermore,  we prove Liouville results for a class of generalized Lane-Emden-Matukuma (LEM) equations, which can be informally considered the non-singular counterpart to the Hardy-Hénon class, as it involves the nonlinearity $f$ of the form $(1+|x|^2)^{-\alpha}u^\gamma$. In the case $\alpha=0$, this equation reduces to the standard Lane-Emden equation. On the other hand, if $\alpha=1$, it becomes the Matukuma equation, which was proposed by the astrophysicist Takehiko Matukuma in 1930 in \cite{M1930}, who, based on his physical intuition, introduced it as a mathematical model to describe the dynamics of globular clusters of stars, see also \cite{M1938}. Consequently, Lane-Emden and Matukuma equations have been extensively investigated by mathematicians as they describe several phenomena in astrophysics and mathematical physics, see \cite{M1930,L1993,Tinh25_1,Tinh25_2,Tinh25_3,LL2023} and references therein for further details. It is remarkable that LEM equations have been comprehensively studied in \cite{Tinh25_2,Tinh25_3} for all positive values of the parameter $\alpha$. Our results in this article regarding these equations also cover a broad range of $\alpha$, including the case $\alpha\in (0,1)$ in \cite{Tinh25_2} and a part of the regime $\alpha>1$ in \cite{Tinh25_3}.  
	
	For exponential nonlinearity (the third type of nonlinearity above), equations involving this type of nonlinearity arise in the theory of gravitational equilibrium of polytropic stars, see \cite{CH1957,JL1972,EG1907}. Therefore, Liouville results for such equations have attracted many researchers over the years. Let us list some typical publications in this field. Firstly, Alarcón, Melián, and Quaas \cite{Quaas2016} established optimal Liouville theorems for supersolutions of the Laplace equation involving exponential nonlinearity, particularly having the form
	\begin{align*}
	-\Delta u=e^u \text{ in an exterior domain of } \mathbb R^n.
	\end{align*}
	Moreover, \cite{Quaas2015} studied Liouville results for the Laplace equation with gradient terms. Additionally, Wang and Ye \cite{WY2012} proposed Liouville type results for the Hénon elliptic problem 
	\begin{align*}
		-\Delta u=|x|^{\alpha}e^u\text{ in } \Omega\subset\mathbb R^n,
	\end{align*}
	whenever $\alpha>-2$ and $n\geq 2$. In this article, we introduce a new result for this class of nonlinearity in the fractional infinity Laplacian framework, which contributes to extending the applicable scope of exponential nonlinearity.
	
	Motivated by these results, we investigate the analogous conditions under which the equation \eqref{main} admits only trivial viscosity solutions. The structure of Hamiltonians in the equation \eqref{main} is inspired by the models considered in \cite{AAE2025}, where nonlocal diffusion and gradient nonlinearities were analyzed in the context of fractional Pucci-type nonlinear operators. Our first contribution is to develop the idea in \cite[Theorem 5.1]{ALT2016} to cover increasing power-type nonlinearities and to include the fractional infinity Laplacian with Hamiltonians above. The fundamental tool of this approach is the comparison principle, which was first applied in \cite{CEF2001} and later enhanced in \cite{Armstrong2011} and \cite{AH2020} for equations with gradient terms and in \cite{MT2017} for weakly coupled systems. It is worth mentioning that we obtain a new Liouville result for equation \eqref{main} with the LEM nonlinearity (type II), for which the methods used in \cite{Tinh25_1,Tinh25_2,Tinh25_3} cannot be applied. Furthermore, the exponential-type nonlinearity $|u|^{\gamma} e^u$ exhibits a much stronger supercritical growth, rendering comparison-principle methods for Liouville-type results inapplicable. Therefore, we overcome this difficulty by adopting the doubling method from  \cite{AAE2025}, which was used there to handle nonlocal equations with fractional Pucci-type operators. This method is widely used in the theory of viscosity solutions to nonlinear equations and we suitably adapt it to obtain regularity and Liouville-type results for these equations.
	
	Regarding regularity, it is one of the primary challenges when studying the infinity Laplace equations. The best regularity results until now are $C^{1,\alpha}$ for infinity harmonic functions in planes proposed by Evans and Savin \cite{ES2008}, and everywhere differentiability for infinity harmonic functions in $\mathbb R^n$, with $n\geq 3$ due to Evans and Smart \cite{ES2011}. For the fractional infinity Laplace equations, Biswas and Vo achieved a notable regularity result in \cite[Lemma 2.2]{AH2020}, namely the locally Lipschitz regularity for solutions to KPP type equations with gradient terms. We extend their ideas to Hardy-Hénon nonlinearities mentioned above, proving in Lemma \ref{lipschitz1} that solutions of \eqref{main} are locally Lipschitz continuous under mild structural conditions on $H$ and $f$. The comparison principle developed in this paper is also influenced by the approaches of \cite[Theorem 2.3]{AH2020}, which combined viscosity solution techniques with carefully designed coupling functions to deal with the degeneracy of the fractional infinity Laplacian and the nonlinearity of gradient terms.
	
	Our weighted version adapts these arguments to the presence of the Hardy-Hénon term $|x|^\alpha$, for which we impose conditions inspired by the recent work of Bezerra Júnior et al. \cite{JSNS2024} to ensure the consistency of the weight in the elliptic setting. Following the weighted Hardy-Hénon framework in \cite{JSNS2024}, we assume that the function $f(|x|,u)$ satisfies the structural growth condition
	\begin{align*}
		f(|x|,1) = \mathcal{H}(x),
	\end{align*}
	where $\mathcal{H}$ acts as a type of weight in our diffusion model and satisfies the following growth condition: there exist positive constants $a_1, a_2$ such that  
	\begin{align*}
		a_1 |x|^\alpha \leq \mathcal{H}(x) \leq a_2 |x|^\alpha.
	\end{align*}  
	This condition allows the exponent $\alpha$ to take negative values, thus including a broader class of Hardy-Hénon. 
	In summary, our contributions can be outlined as follows:
	\begin{itemize}
		\item We prove local Lipschitz regularity estimates for viscosity solutions (Lemma \ref{lipschitz1}), generalizing the techniques of \cite{AH2020} to a broader range of nonlinearities.
		\item We develop comparison principles (Theorem \ref{Comp.Prin}) that unify the approaches of \cite{AH2020,ALT2016} and adapt them to weighted settings inspired by \cite{JSNS2024}.
		\item We establish Liouville-type theorems for certain specific forms of Hardy-Hénon equations, extending the classical absorption framework of \cite{ALT2016} to the fractional infinity Laplacian with gradient (see Theorem \ref{theo: Liouville-type result I}, \ref{theo: Liouville-type result II}, \ref{theo: Liouville-type result - III}) and we partially address the case of an exponential nonlinearity $f(|x|,u)=|u|^{\gamma}e^u$ (Theorem \ref{cxmodel}).
	\end{itemize}
	For clarity, Table \ref{Table_1} provides a summary of the methods (Radial Analysis and Doubling Method), structural conditions, and specific barrier parameters, such as the exponents $p$, scaling constants $\tau$, and dead-core thickness $T$, whose explicit values are calculated in Section \ref{sec: 4}.
	\begin{table}[H]
		\centering
		\normalsize                       
		\setlength{\tabcolsep}{4pt}          
		\renewcommand{\arraystretch}{1.2}    
		\resizebox{1\textwidth}{!}{
			\begin{tabular}{|M{5.5cm}|M{2cm}|M{6cm}|M{4cm}|M{3cm}|M{5cm}|}
				\hline
				Model problem &Method &$h \mbox{ and } T$ & $p$ (or $p^*$) &
				$\tau$ (or $\tau^*$) & Structural assumptions \\
				\hline
				$\Delta_\infty^\beta u=cH(u,\nabla u)+\lambda |x|^\alpha (u^{+})^\gamma$
				&
				\vspace*{3cm}
				Radial Analysis
				&
				\vspace*{-0.3cm}
				\hspace*{-0.3cm}$
				\begin{array}{ll}
					&\bullet \quad \displaystyle h(s)=\tau s^{p}\\[10pt] 
					&\bullet \quad T =\displaystyle\sqrt[1/p]{\dfrac{d}{\tau}}
				\end{array}
				$
				&
				\vspace*{0.2cm}
				\makecell{$\min\{p_1,p_2\}$,\\[3pt] where $p_1,p_2$\\[3pt] are given in \\\eqref{cond_2}}
				\vspace*{0.1cm}
				&
				Given in \eqref{cond_2} 
				&
				\vspace*{0.1cm}
				\makecell{$0<m<3-\beta$,\\ $m + q + \beta < 3$,\\$0\le \gamma < 3-\beta$,\\ $-1-\gamma<\alpha$}\vspace*{0.1cm}
				\\ \cline{1-1}\cline{3-5}\cline{6-6}
				$\Delta_\infty^\beta u = cH(u,\nabla u) + \lambda\left(1 + |x|^2\right)^{-\alpha}u^{\gamma}$
				&
				&
				\vspace*{-0.5cm}
				\makecell{
					\hspace*{-0.2cm}$
					\begin{array}{ll}
						&\bullet \quad \displaystyle h(s)=\chi_{\{p_2^*>p_1^*\}}\tau^*s^{p^*},\\[5pt]
						&\mbox{where } \chi_{\{p_2^*>p_1^*\}} \mbox{ is given in } \eqref{5'}\\[10pt] 
						&\bullet \quad T \text{ solves }\eqref{1'}
					\end{array}
					$
				}
				&
				\vspace*{0.2cm}
				\makecell{$\min\{p_1^*,p_2^*\}$\\[3pt] where $p_1^*,p_2^*$\\[3pt] are given in \\[3pt] \eqref{gradient_term_1} and \eqref{cond:r_finite},\\[3pt]respectively}
				\vspace*{0.1cm}
				&
				Given in \eqref{cond_3}
				&
				\makecell{$0<m<3-\beta$,\\$m + q + \beta < 3$,\\$0\leq\gamma < 3-\beta$,\\ $0\leq\alpha<\frac{(4-\beta-m)(3-\beta-\gamma)}{2(3-\beta-m)}$}
				\vspace*{0.1cm}
				\\ \cline{1-6}
				$\Delta_{\infty}^{\beta}u=a(x)(|x|+1)^{\alpha}|\nabla u|^m+\lambda|u|^{\gamma}e^u$
				&
				Doubling method
				&
				&
				&
				&
				\vspace*{0.2cm}\makecell{$0< m \leq 2-\beta$,\\$0 \leq \gamma$,\\ $-1< \alpha \leq 0$,\\ $\displaystyle \inf_{\mathbb{R}^n}a > 0$}
				\vspace*{0.1cm}\\ 
				\hline
			\end{tabular}
		}
		\caption{Summary of proof methods, key parameters, and structural assumptions for the considered models.}
		\label{Table_1}
	\end{table}
	The rest of the paper is organized as follows. In Section \ref{sec: 2}, we recall the necessary preliminaries and viscosity framework. Section \ref{sec: 3} is devoted to locally Lipschitz estimates and a comparison principle, along with existence and stability results. In Section \ref{sec: 4}, we will analyze the radial boundary value problem. Finally, Section \ref{sec: 5} presents the main Liouville-type theorems along with discussions about the optimality of our results.
	
	
	\section{\bf Preliminaries}
	\label{sec: 2}
	In this section, we introduce some notation, fundamental definitions, as well as several tools to facilitate the proofs of the main results.
	
	Let $\Omega$ be a domain in $\mathbb{R}^n$; we denote by $\partial \Omega$ its boundary.
	We use $B_r(x_0)$ to denote the open ball of radius $r > 0$ centered at $x_0 \in \mathbb{R}^n$, i.e., $B_r(x_0) = \left\{y \in \mathbb{R}^n,|y - x_0| < r \right\}$. When $x_0 = 0$, this ball is denoted by $B_r$. The notation $u \prec_z \varphi$ signifies that $\varphi$ touches $u$ from above exactly at the point $z$, i.e., for some open ball $B_r(z)$ around $z$ we have $u(x) < \varphi(x)$ for $x \in B_r(z) \backslash \{z\}$ and $u(z) = \varphi(z)$. Similarly, $u\succ_z \varphi$ denotes that $\varphi$ touches $u$ from below exactly at the point $z$. Moreover, for any $\Omega \subset \mathbb{R}^n$, we now define the notions of viscosity subsolution and supersolution. It will be useful to have the notations
	\begin{align*}
		&\mbox{USC}(\Omega) = \left\{u:\Omega \to \mathbb{R}; ~u\text{ is upper semicontinuous}\right\},\\
		&\mbox{LSC}(\Omega) = \left\{u:\Omega \to \mathbb{R}; ~u \text{ is lower semicontinuous}\right\}.
	\end{align*}
	To state the results in a general setting, we introduce a Hamiltonian. Let $H: \overline{\Omega}\times \mathbb{R}^n \to \mathbb{R}$ be a continuous function with the following properties: 
	\begin{itemize}
		\item \textbf{Growth condition:} $|H(x,p)| \leq C(1 + |p|^m)$ for some $m \in [0,3-\beta)$ and $(x,p) \in \overline{\Omega}\times\mathbb{R}^n$.
		\item \textbf{Continuous:} $|H(x,p) - H(y,p)| \leq \omega(|x-y|)(1 + |p|^m)$ where $\omega: [0,\infty) \to [0,\infty)$ is a continuous function with $\omega(0) = 0$.
	\end{itemize}
	In this article, we deal with viscosity solutions to equations of the form
	\begin{align}
		\label{gene_main}
		\Delta_\infty^\beta u - cH(u,\nabla u) - \lambda f(|x|,u) = 0 \mbox{ in } \Omega, \mbox{ and } u = g \mbox{ on }\partial \Omega.
	\end{align}
	Here $g$ is assumed to be continuous and $f$ is one of the three types of Hardy-Hénon nonlinearities defined in \eqref{Hf}. Additionally, we use the notation $\mathbb S^{n\times n}$ to denote the set of $n\times n$ real symmetric matrices. Moreover, for a symmetric matrix $A$, we define the maximal and minimal eigenvalues of $A$, respectively, by
	\begin{align}
		\label{matrix_A}
		M(A) = \max_{|x| = 1} \langle x,Ax \rangle, \quad m(A) = \min_{|x| = 1}\langle x,Ax \rangle.
	\end{align}
	\begin{definition}[\bf Viscosity solution]
		\label{Vis_sol}
		An upper-semicontinuous (or lower-semicontinuous) function $u$ on $\overline{\Omega}$ is said to be a viscosity subsolution (or supersolution) of \eqref{gene_main} if the following statements are satisfied:
		\begin{itemize}
			\item[(i)] $u \leq g$ on $\partial \Omega$ ($u \geq g$ on $\partial \Omega$);
			\item[(ii)] If $u \prec_{x_0}\varphi$ ($u\succ_{x_0} \varphi$) for some point $x_0 \in \Omega$ and a $C^2$ test function $\varphi$, then
			\begin{align*}
				&\Delta_\infty^\beta \varphi(x_0) - cH(u(x_0),\nabla \varphi(x_0)) - \lambda f(|x_0|,u(x_0)) \geq 0,\\
				&\left(\Delta_\infty^\beta \varphi(x_0) -cH(u(x_0),\nabla \varphi(x_0)) -\lambda f(|x_0|,u(x_0)) \leq 0\right);
			\end{align*} 
			\item[(iii)] For $\beta = 2$, if $u \prec_{x_0}\varphi$ ($u\succ_{x_0} \varphi$) and $\nabla \varphi(x_0) = 0 $ then
			\begin{align*}
				&M(D^2\varphi(x_0)) - cH(u(x_0),\nabla \varphi(x_0)) -\lambda f(|x_0|,u(x_0)) \geq 0,\\
				&\left(m(D^2\varphi(x_0)) - cH(u(x_0),\nabla \varphi(x_0)) -\lambda f(|x_0|,u(x_0)) \leq 0\right).
			\end{align*}
		\end{itemize}
		We call $u$ a viscosity solution if it is both a subsolution and a supersolution to \eqref{gene_main}.
	\end{definition}
	As is known, one can replace the requirement of strict maximum (or minimum) above by non-strict maximum (or minimum). We also recall the notion of superjet and subjet from \cite{CIL1992}. A second order superjet (subjet) of $u$ at $x_0 \in \Omega$ is defined as
	\begin{align*}
		&J^{2,+}_\Omega u(x_0) = \left\{\left(\nabla \varphi(x_0),D^2\varphi(x_0)\right): \varphi \mbox{ is }C^2 \mbox{ and }u-\varphi \mbox{ has a maximum at }x_0\right\},\\
		&\left(J^{2,-}_\Omega u(x_0) = \left\{\left(\nabla \varphi(x_0),D^2\varphi(x_0)\right): \varphi \mbox{ is }C^2 \mbox{ and }u-\varphi \mbox{ has a minimum at }x_0\right\}\right).
	\end{align*}
	Moreover, another definition of subjet at the point $x_0 \in \Omega$ is the set of all $\left(\nabla \varphi(x_0),D^2\varphi(x_0)\right)$ satisfying the following inequality
	\begin{align}
		\label{subjet}
		u(y) \geq u(x_0) + \langle \nabla \varphi(x_0),y-x_0 \rangle + \frac{1}{2} \langle D^2\varphi( x_0)(y-x_0),y -x_0 \rangle + o(|x_0-y|^2) \mbox{ as }y \to x_0.
	\end{align}
	To define the superjet, we reverse the inequality in \eqref{subjet}. The closure of a superjet is given by
	\begin{align*}
		\overline{J}^{2,+}_\Omega u(x_0) = \left\{(p,X) \in \mathbb{R}^n \times \mathbb{S}^{n\times n}: \exists(p_n,X_n) \in J^{2,+}_{\Omega}u(x_n) \mbox{ such that } \left(x_n,u(x_n),p_n,X_n\right) \to \left(x_0,u(x_0),p,X\right)\right\}.
	\end{align*}
	Similarly, we can also define the closure of a subjet, denoted by $\overline{J}^{2,-}_\Omega u$. See for instance, \cite{CIL1992} for more details.\\ 
		At the end of this section, we will state a lemma concerning the $\mathcal{O}(n)$ invariance of $\Delta_\infty^\beta$ to facilitate the radial symmetry argument in Section \ref{sec: 4}. Here, $\mathcal{O}(n)$ stands for the set of $n\times n$ real orthogonal matrices.
	\begin{lemma}[\bf $\mathcal{O}(n)$ invariance of $\Delta_\infty^\beta$]
		\label{lem:rotation}
		Let $O\in\mathcal{O}(n)$ be an orthogonal matrix, and  set $v(x):=u\left(O(x)\right)$. Then the following identity holds
		\begin{align*}
			\Delta_\infty^\beta v(x) = \Delta_\infty^\beta u\left(O(x)\right),\text{ for all }x\in\mathbb R^n.
		\end{align*}
	\end{lemma}
	\begin{proof}
			We can view $O$ as a linear function, thus its Jacobian is constant and  $O'= O$. Applying the chain rule in vector form $v = u \circ O$ gives us
			\begin{align*}
				\nabla v(x) = O^T \nabla u\left(O(x)\right),\quad x\in\mathbb R^n.
			\end{align*}
		Combining the above identity with the fact that $O$ is an orthogonal matrix, we deduce 
		\begin{align}\label{lemma2.1}
			|\nabla v(x)| = |\nabla u(O(x))|,\quad x\in\mathbb R^n.
		\end{align}
		Also, by the Hessian chain rule and $O'' = 0$, we obtain
		\begin{align*}
			D^2v(x) = O^T\left[D^2u\left( O(x)\right)\right]O,\quad x\in\mathbb R^n.
		\end{align*}
		This yields that for every $x\in\mathbb R^n$, we arrive at
		\begin{align*}
			\langle D^2 v(x)\nabla v(x),\nabla v(x)\rangle = (\nabla v(x))^T (D^2v(x)) \nabla v(x)=\left[O^T\nabla u\left(O(x)\right)\right]^T \left[O^TD^2 u\left(O(x)\right)O\right]\left[O^T \nabla u\left(O(x)\right)\right].
		\end{align*}
		We know that $(AB)^T = B^TA^T$ and $OO^T = I$, then
		\begin{align}\label{lemma2.2}
			\langle D^2 v(x)\nabla v(x),\nabla v(x)\rangle = \langle D^2u\left(O(x)\right)\nabla u\left(O(x)\right),\nabla u\left(O(x)\right)\rangle.
		\end{align}
		Combining \eqref{lemma2.1} with \eqref{lemma2.2}, we easily confirm
		\begin{align*}
			 \Delta_{\infty}^{\beta}v(x)=\Delta_{\infty}^{\beta}u\left(O(x)\right) \text{ in } \mathbb R^n, 
			 \end{align*}
			which completes the proof.
	\end{proof}

	\section{\bf Locally Lipschitz estimates and Comparison principle}
	\label{sec: 3}
	In this section, we state and prove two key results for this paper: improved regularity estimates and a comparison principle. Now, we introduce a lemma about the regularity of viscosity solutions to problem \eqref{main}, which helps us to establish a comparison principle for solutions of this equation. The techniques used to prove the following results are an adaptation of methods in \cite{AH2020} for our fractional infinity Laplacian equations. 
	\begin{lemma}
		\label{lipschitz1}
		Let $f:\Omega\times\mathbb R\to\mathbb R$ satisfy the condition
		\begin{align*}
			\sup\limits_{\Omega\times I}f(|x|,t)<\infty,\text{ for every compact set }I\subset \mathbb R.
		\end{align*}
			Assume that $v$ is a bounded viscosity solution of $\Delta_{\infty}^{\beta}v\geq -c_1|\nabla v|^{m}-c_2 f(|x|,v)$ in $\Omega$ with $c_i\geq 0$ for $i\in\{1,2\}$ and $m \in (0,3-\beta]$. Then $v$ is locally Lipschitz in $\Omega$ with Lipschitz constant depending on $c_i,\|v\|_{L^{\infty}(\Omega)}$, and $\|f_0\|_{L^{\infty}(\Omega)}=\sup\limits_{\Omega}f(|x|,v)$. 	 
	\end{lemma}

	\begin{proof}
		By using Young's inequality, we only need to consider $m=3-\beta$. In this case, $v$ is a solution to
		\begin{align*}
			\Delta_{\infty}^{\beta}v\geq -c_1|\nabla v|^{3-\beta}-c_2f(|x|,v)\text{ in }\Omega.
		\end{align*}		 
		Then we obtain in the viscosity sense 
		\begin{align}
			\label{proof_lips_3}
			\Delta_{\infty}v\geq -\tilde{c}_1|\nabla v|^{3} - \tilde{c}_2,
		\end{align}
		for some nonnegative constants $\tilde{c}_1,\tilde{c_2}$ depending on $c_i$ and $\|f_0\|_{L^{\infty}(\Omega)}$. 
		
		There is no loss of generality in assuming $v\geq 0$. Now, we choose $\Lambda>0$ sufficiently small such that $\Lambda\|v\|_{L^\infty(\Omega)}<1$. We consider the function $h(x)=v(x)+\dfrac{\Lambda}{2}v^2(x)$. 
	A straightforward calculation reveals that
		\begin{align*}
			\Delta_{\infty}h&=(1+\Lambda v)^3\Delta_{\infty}v+\Lambda(1+\Lambda v)^2|\nabla v|^4 \nonumber.
		\end{align*}
	Applying the estimate \eqref{proof_lips_3}, we derive that
		\begin{align}
			\label{proof_lips_4}
			\Delta_{\infty}h\geq (1+\Lambda v)^3\left[-\tilde{c}_1|\nabla v|^{3}-\tilde{c}_2+\dfrac{\Lambda}{1+\Lambda v}|\nabla v|^4\right].
		\end{align}
		Applying Young's inequality again, we deduce
		\begin{align}
			\label{proof_lips_5}
			\tilde{c}_1|\nabla v|^3\le \dfrac{1}{4}(\tilde{c}_1)^4\left(\dfrac{1+\Lambda v}{\Lambda}\right)^3+\dfrac{3}{4}\dfrac{\Lambda}{1+\Lambda v}|\nabla v|^4.
		\end{align}
		Substituting \eqref{proof_lips_5} into \eqref{proof_lips_4}, noting that $1\le 1+\Lambda v<2$, we obtain
		\begin{align*}
			\Delta_{\infty}h&\geq (1+\Lambda v)^3\left[-\dfrac{1}{4}(\tilde{c}_1)^4\left(\dfrac{1+\Lambda v}{\Lambda}\right)^3+\dfrac{1}{4}\dfrac{\Lambda}{1+\Lambda v}|\nabla v|^4-\tilde{c}_2\right]\\
			&\geq -\dfrac{1}{4}(\tilde{c}_1)^4\dfrac{(1+\Lambda v)^6}{\Lambda^3}-\tilde{c}_2.
		\end{align*}
		Since $1\le 1+\Lambda v<2$, it follows
		\begin{align*}
			\Delta_{\infty}h \geq -\left[\dfrac{16}{\Lambda^3}(\tilde{c}_1)^4+8\tilde{c}_2\right].
		\end{align*}
		Therefore, we conclude that $\Delta_{\infty}h\geq -\Gamma$ in $\Omega$, where $\Gamma=\dfrac{16}{\Lambda^3}(\tilde{c}_1)^4+8\tilde{c}_2\geq 0$, from which we conclude that $h$ is locally Lipschitz in $\Omega$ thanks to \cite[Theorem 2.4]{TA2012}. Observe that $v=\dfrac{1}{\Lambda}(\sqrt{1+2\Lambda h}-1)$, and this allows us to complete the proof.
	\end{proof}
	With the local Lipschitz estimate in hand, we next establish a comparison principle
	\begin{theorem}[\bf Comparison principle]
		\label{Comp.Prin}
		Let $\Omega\subset\mathbb{R}^n$ be a bounded domain, $\lambda > 0$. Let $f: \Omega \times \mathbb{R} \to \mathbb{R}$ be a continuous and non-decreasing function in the second variable, in particular, it covers the forms (I) and (II). Suppose $u\in\mathrm{USC}(\overline{\Omega})$ is a viscosity subsolution of
		\begin{align*}
			\Delta_\infty^\beta u - cH(u,\nabla u) - \lambda f(|x|,u) = h_1(x)\text{ in }\Omega,
		\end{align*}
		and $v\in\mathrm{LSC}(\overline{\Omega})$ is a viscosity supersolution of the same equation with $h_1$ replaced by $h_2$. Assume $v\geq u$ on $\partial\Omega$ and $h_1 > h_2$ in $\overline{\Omega}$. Then $v\geq u$ in $\Omega$.
	\end{theorem}
	\begin{proof}
		Let us suppose, for the sake of contradiction, that there exists $A > 0$ such that $A = \displaystyle\sup_{\overline{\Omega}}(u - v)$. Consider
		\begin{align*}
			\omega_\varepsilon(x,y) = u(x) - v(y) - \frac{1}{4\varepsilon}|x-y|^4 \mbox{ for } x,y \in \overline{\Omega}.
		\end{align*}
		For each $\varepsilon > 0$ small, we define the maximum of $\omega_\varepsilon$
		\begin{align*}
			A_\varepsilon = \sup_{\overline{\Omega} \times \overline{\Omega}}\left(u(x) - v(y) - \frac{1}{4\varepsilon}|x-y|^4\right) < \infty.
		\end{align*}
		Note that $A_\varepsilon \geq A$ for all $\varepsilon$. Let $(x_\varepsilon,y_\varepsilon) \in \overline{\Omega} \times \overline{\Omega}$ be a point for $\omega_\varepsilon$ where the maximum is attained. It is then standard to show that \cite[Lemma 3.1]{CIL1992}
		\begin{align*}
			\lim\limits_{\varepsilon \to 0}\frac{1}{4\varepsilon}|x_\varepsilon - y_\varepsilon|^4 = 0, \mbox{ and } \lim\limits_{\varepsilon \to 0}A_\varepsilon = A.
		\end{align*}
		In particular, there exists $z_0 \in \Omega$ such that
		\begin{align}
			\label{proof_2}
			\lim\limits_{\varepsilon \to 0}x_\varepsilon = \lim\limits_{\varepsilon \to 0}y_\varepsilon = z_0,\text{ and }u(z_0)-v(z_0)=A.
		\end{align}
		Moreover, one observes that
		\begin{align*}
			A > 0 \geq \sup_{\partial \Omega}(u - v),
		\end{align*}
		i.e, the maximizer cannot move towards the boundary. As a result, $x_\varepsilon, y_\varepsilon \in \Omega_1$ for some interior domain $\Omega_1 \Subset \Omega$ and $\varepsilon > 0$ sufficiently small. Since $u,v$ are Lipschitz continuous in $\Omega_1$~(Lemma \ref{lipschitz1}), there exists $L > 0$ such that
		\begin{align}\label{com4}
			|u(z_1) - u(z_2)| + |v(z_1) - v(z_2)| \leq L |z_1 - z_2|, \quad z_1, z_2 \in \Omega_1.
		\end{align}
		From $\omega_\varepsilon(x_\varepsilon, x_\varepsilon) \leq \omega_\varepsilon(x_\varepsilon, y_\varepsilon)$, one obtains that
		\begin{align}\label{com5}
		\frac{1}{4\varepsilon} |x_\varepsilon - y_\varepsilon|^4\leq v(x_\varepsilon) - v(y_\varepsilon).
		\end{align}
		Combining \eqref{com4} with \eqref{com5}, we deduce that
		\begin{align}
			\label{proof_3}
			|x_\varepsilon - y_\varepsilon|^3 \leq 4\varepsilon L.
		\end{align}
		Set $\eta_\varepsilon = \frac{1}{\varepsilon} |x_\varepsilon - y_\varepsilon|^2 (x_\varepsilon - y_\varepsilon) \mbox{ and } \Phi_\varepsilon(x,y) = \frac{1}{4\varepsilon}|x-y|^4$. Following \cite[Theorem 3.2]{ALT2016}, there exist symmetric matrices $X, Y \in \mathbb{S}^{n \times n}$ such that
		\begin{align*}
			(\eta_\varepsilon, X) \in \overline{J}^{2,+}_{\Omega} u(x_\varepsilon), \quad (\eta_\varepsilon, Y) \in \overline{J}^{2,-}_{\Omega} v(y_\varepsilon),
		\end{align*}
		and 
		\begin{align}
			\label{matrix}
			\left(\begin{array}{lc}
				X &0\\
				0 &-Y
			\end{array}\right) \leq D^2\Phi_\varepsilon(x_\varepsilon,y_\varepsilon) + \varepsilon \left[D^2\Phi_\varepsilon(x_\varepsilon,y_\varepsilon)\right]^2.
		\end{align}
		In particular, we get $X \leq Y$. If $\eta_{\varepsilon}\neq 0$, from the viscosity subsolution property of $u$, it gives us
		\begin{align}\label{com1}
			h_1(x_\varepsilon) &\leq |\eta_\varepsilon|^{-\beta} \langle X \eta_\varepsilon, \eta_\varepsilon \rangle - cH(u(x_{\varepsilon}),\eta_\varepsilon) - \lambda f(|x_\varepsilon|,u(x_\varepsilon))\notag\\
			&\leq |\eta_\varepsilon|^{-\beta} \langle Y \eta_\varepsilon, \eta_\varepsilon \rangle - cH(u(x_{\varepsilon}),\eta_\varepsilon) - \lambda f(|x_\varepsilon|,u(x_\varepsilon)).
		\end{align}
		Similarly, applying the definition of supersolution, we obtain
		\begin{align}\label{com2}
			h_2(y_\varepsilon) &\geq |\eta_\varepsilon|^{-\beta} \langle Y \eta_\varepsilon, \eta_\varepsilon \rangle- cH(v(y_{\varepsilon}),\eta_\varepsilon) - \lambda f(|y_\varepsilon|,v(y_\varepsilon)).
		\end{align}
		From \eqref{com1} and \eqref{com2}, we have the following estimate
		\begin{align}\label{com3}
			h_1(x_\varepsilon) &\leq  h_2(y_\varepsilon) +cH(v(y_{\varepsilon}),\eta_\varepsilon)- cH(u(x_{\varepsilon}),\eta_\varepsilon) + \lambda f(|y_\varepsilon|,v(y_\varepsilon)) - \lambda f(|x_\varepsilon|,u(x_\varepsilon)).
		\end{align}
		From the condition on $H$, \eqref{com3} implies that 
		\begin{align*}
			h_2(y_\varepsilon) - h_1(x_\varepsilon)+ \lambda\left[f(|y_\varepsilon|,v(y_\varepsilon)) - f(|x_\varepsilon|,u(x_\varepsilon))\right] + c\omega(|x_\varepsilon-y_\varepsilon|)(1+ |\eta_\varepsilon|^m)\geq 0.
		\end{align*}
		Letting $\varepsilon \to 0$, using \eqref{proof_2} and \eqref{proof_3}, we conclude that
			\begin{align}
			\label{result_1}
			\sup_{\overline{\Omega}}\left(h_2 - h_1\right) \geq \lambda\left[f(|z_0|,u(z_0)) - f(|z_0|,v(z_0))\right].
		\end{align}
		Since $u(z_0)>v(z_0)$ and $f$ is non-decreasing in the second variable, then
		\begin{align*}
			\sup_{\overline{\Omega}}\left(h_2 - h_1\right)\geq 0,
		\end{align*}
		which contradicts the assumption $h_1>h_2$ on $\overline{\Omega}$.
		
		Otherwise, if $\eta_\varepsilon = 0$ then $x_\varepsilon = y_\varepsilon$. Therefore, from \eqref{matrix}, it yields 
		\begin{align*}
			\left(\begin{array}{lc}
				X &0\\
				0 &-Y
			\end{array}\right) \leq \left(\begin{array}{lc}
				0 &0\\
				0 &0
			\end{array}\right),
		\end{align*}
		which leads to $X \leq 0 \leq Y$. Combining this with \eqref{matrix_A}, we infer that
		\begin{align}
			\label{Cond_M}
			M(X) \leq 0 \leq m(Y).
		\end{align}
		For $\beta\in [0,2)$, we can apply similar techniques as the above case to get a similar contradiction. Finally, we address the case $\beta=2$ and $\eta_\varepsilon = 0$. By Definition \ref{Vis_sol} for this case, we have
		\begin{align*}
			&\mbox{Subsolution: } M(X) - cH(u(x_\varepsilon),0) - \lambda f(|x_\varepsilon|,u(x_\varepsilon)) \geq h_1(x_\varepsilon);\\
			&\mbox{Supersolution: }m(Y) - cH(v(y_\varepsilon),0) - \lambda f(|y_\varepsilon|,v(y_\varepsilon)) \leq h_2(y_\varepsilon).
		\end{align*}
		Subtracting the subsolution inequality from the supersolution inequality, we obtain
		\begin{align*}
			h_2(y_\varepsilon) - h_1(x_\varepsilon) \geq m(Y) - M(X) +c\left[H(u(x_\varepsilon),0)-H(v(y_\varepsilon),0)\right] - \lambda f(|y_\varepsilon|,v(y_\varepsilon)) + \lambda f(|x_\varepsilon|,u(x_\varepsilon)).
		\end{align*}
		From \eqref{Cond_M}, we readily obtain
		\begin{align}
			\label{case:eta = 0-1}
			h_2(y_\varepsilon) - h_1(x_\varepsilon) \geq  \lambda f(|x_\varepsilon|,u(x_\varepsilon))- \lambda f(|y_\varepsilon|,v(y_\varepsilon)). 
		\end{align}
		Letting $\varepsilon \to 0$ in \eqref{case:eta = 0-1}, we get the estimate \eqref{result_1} again, which is a contradiction. The proof is completed.
	\end{proof}
		As a consequence, we obtain existence and uniqueness of solutions under suitable conditions. This result hinges on the possibility of sandwiching the solution between two special functions, known as a supersolution and a subsolution. Specifically, we demonstrate that if a pair of ordered supersolution and subsolution can be constructed that satisfy appropriate regularity and boundary conditions, then the equation admits a unique solution within the region delimited by this pair.
	\begin{theorem}[\bf Existence and Uniqueness]
		\label{theo:sandwich}
		Let $\Omega \subset \mathbb{R}^n$ be a bounded $C^1$ domain. Suppose that $h$ and $g$ are continuous functions in $\overline{\Omega}$. Let  $f:\Omega\times\mathbb R\to\mathbb R$ be a continuous function. We consider the boundary value problem
		\begin{align}
			\label{barrel-method}
			\left\{
			\begin{array}{clcl}
				\Delta_{\infty}^{\beta} u - cH(u,\nabla u) -  \lambda f(|x|, u) &= &h(x) &\text{in } \Omega, \\
				u &= &g &\text{on } \partial\Omega.
			\end{array}\right.
		\end{align}
		Assume that $\overline{u} \in C(\overline{\Omega})$ is a supersolution and $\underline{u} \in C(\overline{\Omega})$ is a subsolution to the problem \eqref{barrel-method} such that $\overline{u} \geq \underline{u} \geq 0$. We also suppose that $\underline{u} = g$ on $\partial\Omega$. Then, the equation \eqref{barrel-method} admits a unique solution $u$ satisfying the enclosure
		\begin{align*}
			0\leq\underline{u} \leq u \leq \overline{u} \text{ in } \overline{\Omega}.
		\end{align*}
	\end{theorem}
	\begin{proof}
		The proof for uniqueness, on the other hand, relies entirely on Theorem \ref{Comp.Prin}. The existence result is obtained via the established Perron's method, coupled with the requirement for a carefully constructed barrier function near the boundary. A brief outline of the proof is included here. We introduce the set $\mathcal{T}$, defined as the collection of all supersolutions that are bounded above by $\overline{u}$
		\begin{align*}
			\mathcal{T} = \left\{\nu \in \mbox{LSC}(\overline{\Omega}):~\nu \mbox{ is a supersolution to }\eqref{barrel-method} \mbox{ and }\nu \leq \overline{u}\right\}.
		\end{align*}
		Also, we put $v(x) = \displaystyle \inf_{\nu \in \mathcal{T}}\nu(x)$ and
		\begin{align*}
			v^\ast(x) = \lim\limits_{r \to 0} ~\inf\left\{v(y):~y\in\overline{\Omega},|x-y| \leq r\right\},
		\end{align*}
		be the lower semicontinuous (LSC) envelope of $v$ (see \cite[Page 22]{CIL1992}), ensuring that $v^\ast$ is  the largest LSC function less than or equal to $v$. Because every function in $\mathcal{T}$ is a supersolution, $v$ satisfies the boundary condition $v \geq g$ on $\partial \Omega$. The primary goal is to show that $v^\ast$ belongs to the set $\mathcal{T}$ which would then imply $v = v^\ast$. By Lemma \ref{lipschitz1}, we also deduce $v$ is locally Lipschitz continuous in $\Omega$. Now suppose that $v^\ast$ is not a supersolution. Assume that for some $\varphi \in C^2(\Omega)$, we have $\varphi \prec_{x_0} v^\ast$ for some $x_0 \in \Omega$ and
		\begin{align*}
			\Delta_{\infty}^{\beta} \varphi(x_0) - cH(\varphi(x_0), \nabla \varphi(x_0)) - \lambda f(|x_0|, \varphi(x_0)) > h(x_0).
		\end{align*}
		Using continuity we can find a ball $B(x_0)$ around $x_0$ satisfying
		\begin{align}
			\label{proof:barr1}
			\Delta_{\infty}^{\beta} \varphi(x) - cH(\varphi(x), \nabla \varphi(x)) -\lambda f(|x|, \varphi(x)) > h(x)\mbox{ in }\overline{B(x_0)}.
		\end{align}
		Now, for every $\varepsilon >0$ we can find a pair $\left(\nu_\varepsilon,x_\varepsilon\right) \in \mathcal{T} \times B(x_0)$ satisfying
		\begin{align*}
			\nu_\varepsilon(x_\varepsilon) - \varphi(x_\varepsilon) = \displaystyle\inf_{B(x_0)}\left(\nu_\varepsilon-\varphi\right) < \varepsilon.
		\end{align*}
		Note that $x_\varepsilon \to x_0$ as $\varepsilon \to 0$. Also, $\varphi + \nu_\varepsilon(x_\varepsilon) - \varphi(x_\varepsilon)$ touches $\nu_\varepsilon$ at $x_\varepsilon$ from below. Hence, by the definition of supersolution we must have
		\begin{align*}
			\Delta_\infty^\beta \varphi(x_\varepsilon) -cH(\varphi(x_{\varepsilon}),\nabla \varphi(x_\varepsilon)) -\lambda f(|x_\varepsilon|,\nu_\varepsilon(x_\varepsilon)) \leq h(x_\varepsilon),
		\end{align*}
		and letting $\varepsilon \to 0$, we obtain a contradiction to \eqref{proof:barr1}. To fully establish this claim, the only remaining step is to demonstrate that $v^\ast \geq g$ on $\partial \Omega$. Thus the boundary condition is satisfied because, by Theorem \ref{Comp.Prin}, every function $\nu$ in the set $\mathcal{T}$ is greater than or equal to $\underline{u}$, which in turn satisfies the condition $\nu \geq g$ on $\partial \Omega$. Therefore, we conclude that both $v \geq \underline{u}$ and $v^\ast \geq g$ on $\partial \Omega$.\\
		We omit the detailed proof that $v$ is a subsolution in $\Omega$, as this is a standard result in the literature. The interested reader may refer to the arguments detailed in \cite[Theorem 1]{LW20081}. To complete the proof we must check that
		\begin{align*}
			\lim\limits_{x \in \Omega \to z}v(x) = g(z)\text{ for all }z \in \Omega. 
		\end{align*}
		Pick $z \in \partial \Omega$. To begin, we extend the boundary data $g$ continuously to $\mathbb{R}^n$. For a given $\varepsilon > 0$, we select a radius $r > 0$ small enough such that 
		\begin{align*}
			|g(x) - g(z)| \leq \varepsilon\mbox{ for }x \in B_r(z).
		\end{align*}
		We now proceed to construct a barrier function. Fix any $r_1 \in (0,r \wedge 1)$ where $r \wedge 1 = \min(r,1)$ and we also define the auxiliary function
		\begin{align}
			\label{proof:barr2}
			\chi(x) = |x|^\alpha - r_1^\alpha\mbox{ for some }\alpha \in \left(0,\frac{m}{m+q}\right),
		\end{align}
		is a classical barrier employed to establish boundary regularity via the comparison principle. The choice of a radial power-law form, $|x|^\alpha$, is motivated by the rotational symmetry of the infinity Laplacian operator. The constant term $-r_1^\alpha$ is subsequently subtracted to enforce the geometric condition that the function vanishes on the sphere $\partial B_{r_1}(z)$, i.e., $\chi(x) = 0$ for $|x| = r_1$. Next, we calculate the following function for the specific forms of Hamiltonians $H$
		\begin{align}
			\label{proof:barr3}
			\Delta_\infty^\beta \chi- cH(\chi,\nabla \chi).
		\end{align}
		Substituting \eqref{proof:barr2} into \eqref{proof:barr3} with each specific form of the operator $H$, we obtain for each $x\in\Omega$
		\begin{align}
			\label{proof:barr4}
			\begin{split}
				&|\nabla \chi(x)|^{-\beta}\langle D^2\chi(x)\nabla \chi(x),\nabla \chi(x) \rangle - c|\nabla\chi(x)|^m,\\[5pt]
				&|\nabla \chi(x)|^{-\beta}\langle D^2\chi(x)\nabla \chi(x),\nabla \chi(x) \rangle +c(\chi(x))^q|\nabla\chi(x)|^m, \\[5pt]
				&|\nabla \chi(x)|^{-\beta}\langle D^2\chi(x)\nabla \chi(x),\nabla \chi(x) \rangle - c(\chi(x))^q|\nabla\chi(x)|^m.
			\end{split} 
		\end{align}
		On the other hand, a simple calculation shows that
		\begin{align*}
			\nabla\chi(x) = \alpha |x|^{\alpha-2}x \text{ and } D^2\chi(x) = \alpha|x|^{\alpha-2}I + \alpha(\alpha-2)|x|^{\alpha-4}xx^T.
		\end{align*}
		Continuing the analysis, we obtain
		\begin{align*}
			&|\nabla \chi(x)| = \alpha|x|^{\alpha - 1},\\[3pt]
			&D^2\chi(x)\nabla \chi(x) = \left(\alpha|x|^{\alpha-2}I + \alpha(\alpha-2)|x|^{\alpha-4}xx^T\right)\left(\alpha |x|^{\alpha-2}x\right)= \alpha^2(\alpha-1)|x|^{2\alpha - 4}x,\\[3pt]
			&\langle D^2\chi(x)\nabla \chi(x),\nabla \chi(x) \rangle = \left(\alpha^2(\alpha-1)|x|^{2\alpha - 4}x\right)\left(\alpha |x|^{\alpha-2}x\right) = \alpha^3(\alpha - 1)|x|^{3\alpha - 4}.
		\end{align*}
		This allows us to deduce
		\begin{align*}
			|\nabla \chi(x)|^{-\beta}\langle D^2\chi(x)\nabla \chi(x),\nabla \chi(x) \rangle = \alpha^{3-\beta}|x|^{(3-\beta)\alpha-4+\beta}(\alpha-1).
		\end{align*}
		As a result, three expressions in \eqref{proof:barr4} become
		\begin{align*}
			\begin{split}
				&\alpha^{3-\beta}|x|^{(3-\beta)\alpha-4+\beta}(\alpha-1) - c\alpha^m |x|^{(\alpha-1)m},\\[5pt]
				&\alpha^{3-\beta}|x|^{(3-\beta)\alpha-4+\beta}(\alpha-1) +c\left(|x|^\alpha - r_1^\alpha\right)^q\alpha^m |x|^{(\alpha-1)m}, \\[5pt]
				&\alpha^{3-\beta}|x|^{(3-\beta)\alpha-4+\beta}(\alpha-1) - c\left(|x|^\alpha - r_1^\alpha\right)^q\alpha^m |x|^{(\alpha-1)m}, 
			\end{split} 
		\end{align*}
		for $|x| > r_1$. With $m < 3 -\beta$ and given the pre-specified condition $\alpha < \frac{m}{m+q} < 1$, since $q > 0$, and the exponent of $|x|$ can be easily analyzed as follows
		\begin{align*}
			(3-\beta)\alpha-4+\beta = 3\alpha -4 - \beta(\alpha-1) < (3\alpha -3) - \beta(\alpha - 1) = (3-\beta)(\alpha-1) < 0,
		\end{align*}
		and
		\begin{align*}
			(\alpha - 1)m < 0,~~\alpha q + (\alpha - 1)m <0.
		\end{align*}
		We can choose $r_1,\psi > 0$ small so that for $r_1 \leq |x| \leq r_1 + \psi$, which allows us to obtain
		\begin{align*}
			\Delta_\infty^\beta \left(\zeta\chi(x)\right) - cH(\zeta\chi(x),\zeta\nabla \chi(x)) - \lambda f(|x|,\zeta \chi(x)) \leq -\|h\|_{L^\infty(\Omega)}.
		\end{align*}
		for some large $\zeta$. To handle the geometry near the boundary point $z$, we rotate the domain $\Omega$ such that the ball $B_{r_1}$ is tangent to $\partial \Omega$ at $z$ from the exterior. This arrangement ensures that for any $x \in \Omega$ in a neighborhood of $z$, we have $|x| > r_1$. Consequently, our auxiliary function \eqref{proof:barr2} is strictly positive in this region. Next, we introduce the function
		\begin{align*}
			w = \min\left\{g(z) + \varepsilon + \zeta \chi,\overline{u}\right\}.
		\end{align*}
		In fact, it is also a supersolution, and hence $v \leq w$ in $\Omega$, we take the limit as $x \to z$. Since $z$ is the point of tangency for the exterior ball $B_{r_1}$, which implies $\chi(z) = 0$. Thus, it follows that  
		\begin{align*}
			\lim\limits_{x \to z}\sup v(x) \leq \lim\limits_{x \to z}w(x)  = g(z) + \varepsilon.
		\end{align*}
		Since this holds for any $\varepsilon >0$, we have $\lim\limits_{x \to z}\sup v(x) \leq g(z)$. A similar argument using a subsolution yields $\lim\limits_{x \to z}\inf v(x) \geq g(z)$, which completes the proof.
	\end{proof} 
		Finally, we present an extended version of \cite[Corollary 2.6]{Armstrong2011}. The proof follows by applying analogous techniques to those employed in \cite{Armstrong2011} for equations involving the infinity Laplacian and gradient terms. This result will play a crucial role in establishing Theorem \ref{theo: Liouville-type result I} and Theorem \ref{theo: Liouville-type result II} in the subsequent section.
	\begin{lemma}[\bf Stability]
		\label{lem_Stab}
		Let $\left\{\xi_k\right\}_{k=1}^\infty$ be a sequence of lower-semicontinuous functions on $\Omega$ such that $|\xi_j|\le K$ for all $j\geq 1$ and some $K>0$. Suppose that $\xi_k\to \xi$ locally uniformly in $\Omega$ for some $\xi\in \mathrm{LSC}(\overline{\Omega})$, and for each $k$, $u_k$ is a viscosity solution of the following Dirichlet problem
		\begin{align*}
			\begin{split}
				\left\{\begin{array}{clcl}
					\Delta_{\infty}^{\beta}u_k&=&c|\nabla u_k|^m+ \xi_k&\text{in } \Omega, \\[5pt]
					u_k &= &g&\text{on } \partial \Omega.
				\end{array}\right.
			\end{split}
		\end{align*}
		Then $\left\{u_k\right\}$ has a subsequence that converges locally uniformly in $\Omega$ to a solution $u$ of the Dirichlet problem	
		\begin{align*}
			\begin{split}
				\left\{\begin{array}{clcl}
					\Delta_{\infty}^{\beta}u&=&c|\nabla u|^m+\xi&\text{in } \Omega, \\[5pt]
					u&=&g&\text{on }\partial \Omega.
				\end{array}\right.
			\end{split}
		\end{align*}
	\end{lemma}

	\section{\bf Radial analysis}
	\label{sec: 4}
	\noindent In this intermediate section, we pause to analyze the radial boundary value problem
	\begin{align}
		\label{gene_main_rad}
		\tag{A}
		\left\{
		\begin{array}{cccc}
			\Delta_\infty^\beta u &= &cH(u,\nabla u) + \lambda f(|x - x_0| -\rho,u) &\mbox{in } B_R(x_0),\\[5pt]
			u &= &d &\mbox{on }\partial B_R(x_0),
		\end{array}
		\right.
	\end{align}
	where $\beta\in [0,2]$, the parameters satisfy $\lambda >0,d > 0$ and the constant $c$ is chosen depending on $H$ such that the quantity $cH$ remains positive, and the nonlinearities are modeled as in \eqref{Hf}.\\
	The direct method to solve this problem was introduced and analyzed in the paper "Infinity Laplacian equation with strong absorptions" by Araújo, Leitão, and Teixeira \cite[Section 5]{ALT2016}, where the authors treated the infinity Laplace equation $\Delta_\infty u=\lambda(u^+)^{\gamma}$ and they employed a specific line of reasoning to demonstrate the radial symmetry of $u$ via uniqueness and $\mathcal{O}(n)$ invariance. Similarly, for our class of problems, by uniqueness and Lemma \ref{lem:rotation}, it is plain that the solution of such a boundary value problem is radially symmetric. Indeed, for any $O \in \mathcal{O}(n)$, the function $v(x - x_0) := u(O(x-x_0))$ solves the same boundary value problem. Combined with the uniqueness of solutions, this yields $v(x) = u(x)$. Since $O \in \mathcal{O}(n)$ was taken arbitrary, it follows that $u$ is radially symmetric.\\
	On the other hand, let us put $u(x)=h(r)$ with $r=|x-x_0|$. We calculate the gradient and the Hessian matrix of $u$ for $x\neq x_0$ as follows 
	\begin{align*}
		\nabla u(x) = h'(r)e_r; ~
		D^2u(x)=h''(r)(e_r e_r^T)+\frac{h'(r)}{r}\big(I-e_r e_r^T\big).
	\end{align*}
	where $e_r=\dfrac{x-x_0}{|x-x_0|}$. 
	Hence, we have the identity
	\begin{align*}
		\Delta_\infty^\beta u(x) = |h'(r)|^{2-\beta}h''(r),\text{ for all }x\neq x_0.
	\end{align*}
	Note that in our barrier construction with a dead-core, we have $h'(r)\ge0$ (in particular at $r=\rho$). Thus we can write $|h'|^{2-\beta} = (h')^{2-\beta}$.\\
	To begin with form (I) of the Hardy-Hénon type equations, the system \eqref{gene_main_rad} can be rewritten as
	\begin{align}
		\label{gene_main_rad1}
		\tag{A1}
		\left\{
		\begin{array}{clcl}
			\Delta_\infty^\beta u &= &cH(u,\nabla u) + \lambda \left(|x - x_0| -\rho\right)^\alpha (u^+)^\gamma &\mbox{in } B_R(x_0),\\[5pt]
			u &= &d &\mbox{on }\partial B_R(x_0).
		\end{array}
		\right.
	\end{align}
	Equivalently, the system \eqref{gene_main_rad1} can be divided into four cases
	\begin{align}
		\label{rad_2_main}
		\begin{split}
			&\left\{
			\begin{array}{clcl}
				\Delta_\infty^\beta u &= &\lambda \left(|x - x_0| -\rho\right)^\alpha (u^+)^\gamma &\mbox{in } B_R(x_0),\\[5pt]
				u &= &d &\mbox{on }\partial B_R(x_0).
			\end{array}
			\right.\\
			&\left\{
			\begin{array}{clcl}
				\Delta_\infty^\beta u &= &c|\nabla u|^m + \lambda \left(|x - x_0| -\rho\right)^\alpha (u^+)^\gamma &\mbox{in } B_R(x_0),\\[5pt]
				u &= &d &\mbox{on }\partial B_R(x_0).
			\end{array}
			\right.\\
			&\left\{
			\begin{array}{clcl}
				\Delta_\infty^\beta u &= &-cu^q|\nabla u|^m + \lambda \left(|x - x_0| -\rho\right)^\alpha (u^+)^\gamma &\mbox{in } B_R(x_0),\\[5pt]
				u &= &d &\mbox{on }\partial B_R(x_0).
			\end{array}
			\right.\\
			&\left\{
			\begin{array}{clcl}
				\Delta_\infty^\beta u &= &cu^q|\nabla u|^m + \lambda \left(|x - x_0| -\rho\right)^\alpha (u^+)^\gamma &\mbox{in } B_R(x_0),\\[5pt]
				u &= &d &\mbox{on }\partial B_R(x_0).
			\end{array}
			\right.
		\end{split}
	\end{align}
	where $d,q$ and $m$ are positive constants, $0 < \rho < R$, and $x_0 \in \mathbb{R}^n$. Here we allow the Thiele modulus $\lambda > 0$ to be arbitrary, in order to broaden the scope of our analysis. Next, we consider the following ODE related to \eqref{rad_2_main}:
	\begin{align}
		\label{eq: rad_2_1}
		\begin{array}{lll}
			h''(s)(h'(s))^{2-\beta} &= \lambda s^\alpha (h(s)^+)^\gamma&\mbox{in } (0,T),\\[10pt]
			h''(s)(h'(s))^{2-\beta} &= c|h'(s)|^m + \lambda s^\alpha (h(s)^+)^\gamma &\mbox{in } (0,T),\\[10pt]
			h''(s)(h'(s))^{2-\beta} &= -c(h(s))^q|h'(s)|^m + \lambda s^\alpha (h(s)^+)^\gamma &\mbox{in } (0,T),\\[10pt]
			h''(s)(h'(s))^{2-\beta} &= c(h(s))^q|h'(s)|^m+ \lambda s^\alpha (h(s)^+)^\gamma &\mbox{in } (0,T),\\
		\end{array}
	\end{align}
	satisfying the initial condition: $h(0) = 0$ and $h(T) = d$. To solve \eqref{eq: rad_2_1}, let $h(s) = \tau s^p$, and importantly, for $h(s)$ to have zero gradient at the core (i.e., to exhibit a "dead-core" or "plateau"), we require $h'(0^+) = 0$. Now, we have the derivatives $h'$ and $h''$ as follows: 
	\begin{align*}
		h'(s) = \tau ps^{p-1} \mbox{ and } h''(s) = \tau p(p-1)s^{p-2}.
	\end{align*}
	Note that the condition for this derivative to approach $0$ as $s \to 0$ is precisely $p > 1$. Furthermore, it is necessary to have $\tau > 0$. By direct calculations, the four forms in \eqref{eq: rad_2_1} become
	\begin{align}
		\label{proof_rad_2_1}
		\begin{array}{ll}
			\tau^{3-\beta} p^{3-\beta}(p-1) s^{(p-2) + (p-1)(2-\beta)} &= \lambda \tau^\gamma  s^{p\gamma + \alpha},\\[10pt]
			\tau^{3-\beta} p^{3-\beta}(p-1) s^{(p-2) + (p-1)(2-\beta)} &= c\tau^{m}p^{m}s^{(p-1)m} + \lambda \tau^\gamma  s^{p\gamma + \alpha},\\[10pt]
			\tau^{3-\beta} p^{3-\beta}(p-1) s^{(p-2) + (p-1)(2-\beta)} &= -c\tau^{q+m}s^{pq+(p-1)m}p^m + \lambda \tau^\gamma  s^{p\gamma + \alpha},\\[10pt]
			\tau^{3-\beta} p^{3-\beta}(p-1) s^{(p-2) + (p-1)(2-\beta)} &= c\tau^{q+m}s^{pq+(p-1)m}p^m+ \lambda \tau^\gamma  s^{p\gamma + \alpha}.
		\end{array}
	\end{align}
	With the first equation in $\eqref{proof_rad_2_1}$, we can easily balance the powers of $s$ between the left-hand side (LHS) and the right-hand side (RHS) as follows
	\begin{align}\label{system}
		\left\{
		\begin{array}{ccc}
			(p-2) + (p-1)(2-\beta) &= &p\gamma + \alpha, \\[5pt]
			\tau^{3-\beta} p^{3-\beta}(p-1) &= &\lambda \tau^\gamma.
		\end{array}
		\right.
	\end{align} 
	Solving the system \eqref{system}, we have
	\begin{align}
		\label{radial_sol_c=0}
		\left\{
		\begin{array}{cll}
			p &= \displaystyle\frac{4 - \beta + \alpha}{3 - \beta - \gamma},\\[5pt] 
			\tau(\lambda,\gamma,\beta,\alpha) &= \sqrt[3-\beta-\gamma]{\lambda\cdot\frac{ (3-\beta-\gamma)^{4-\beta}}{(4-\beta+\alpha)^{3-\beta}(1+\alpha+\gamma)}},
		\end{array}
		\right., \quad \text{with}~\beta + \gamma < 3,~ \alpha + \gamma > -1. 
	\end{align}
	Applying the boundary condition $h(T) = d$ such that
	\begin{align*}
		d = \tau T^p \implies T = \left(\frac{d}{\tau}\right)^{1/p} = \left(\frac{d}{\tau(\lambda,\gamma,\beta,\alpha)}\right)^{\frac{3 - \beta - \gamma}{4 - \beta + \alpha}}.
	\end{align*}
	In summary, we obtain the solution 
	\begin{align*}
		h(s) = \tau(\lambda,\gamma,\beta,\alpha) \cdot s^{\frac{4-\beta + \alpha}{3 - \beta - \gamma}},s\in (0,T].
	\end{align*}
	To obtain an exact radial balance, we would need the same power of $s$ on all terms. In general, it is impossible for three different exponents, thus the natural strategy is to treat the two right-hand terms separately and obtain two critical profiles: one where the potential term $\lambda \left(|x - x_0| -\rho\right)_{+}^\alpha (u^+)^\gamma$ dominates, and one where the gradient term $|\nabla u|^{m}$ or $\pm u^q|\nabla u|^{m}$ dominates.\\
	Firstly, we will match the LHS term with the gradient term: solving $\eqref{proof_rad_2_1}_2$
	\begin{align}
		\label{proof_rad_2_1_a}
		\tag{a}
		\left\{
		\begin{array}{ccc}
			(p_1-2) + (p_1-1)(2-\beta) &= &(p_1-1)m, \\[5pt]
			\tau_1^{3-\beta} p_1^{3-\beta}(p_1-1) &= &c\tau_1^{m}p_1^{m}.
		\end{array}
		\right.
	\end{align}
	Next, we have a corresponding $\eqref{proof_rad_2_1}_3$ 
	\begin{align}
		\label{proof_rad_2_1_b}
		\tag{b}
		\left\{
		\begin{array}{ccc}
			(p_1-2) + (p_1-1)(2-\beta) &= &(p_1-1)m + p_1q, \\[5pt]
			\tau_1^{3-\beta} p_1^{3-\beta}(p_1-1) &= &-c\tau_1^{q+m}p_1^{m}.
		\end{array}
		\right.
	\end{align}
	And the final equation $\eqref{proof_rad_2_1}_4$ is that
	\begin{align}
		\label{proof_rad_2_1_c}
		\tag{c}
		\left\{
		\begin{array}{ccc}
			(p_1-2) + (p_1-1)(2-\beta) &= &(p_1-1)m + p_1q,\\[5pt]
			\tau_1^{3-\beta} p_1^{3-\beta}(p_1-1) &= &c\tau_1^{q+m}p_1^{m}. 
		\end{array}
		\right.
	\end{align}
	Combining $h(T) = d$, we collect three solutions
	\begin{align}
		\label{gradient_term_1}
		|\nabla u|^m \text{ then solution is } \left\{
		\begin{array}{cll}
			p_1 &= \displaystyle\frac{4 - m -\beta}{3 - m - \beta},\\[10pt] 
			\tau_1(c,\beta,m) &=  \displaystyle\sqrt[3 - m - \beta]{\frac{c\left(3- m - \beta\right)^{4 - m - \beta}}{(4 - m - \beta)^{3- m - \beta}}},\\[10pt] 
			T_1 &=  \displaystyle\left(\frac{d}{\tau_1(c,\beta,m)}\right)^{\frac{3 - m - \beta}{4 - m -\beta}},
		\end{array}
		\right. \quad m +\beta < 3, c > 0,
	\end{align}
	and
	\begin{align}
		\label{gradient_term_2}
		-u^q|\nabla u|^m \text{ then solution is }\left\{
		\begin{array}{cll}
			p_1 &= \displaystyle\frac{4 -\beta - m}{3 - m - \beta - q},\\[10pt] 
			\tau_1(c,\beta,m,q) &=  \displaystyle\sqrt[3 - m - \beta - q]{\frac{-c\left(3- m - \beta-q\right)^{4 - m - \beta}}{(4 - m - \beta)^{3- m - \beta-q}}},\\[10pt] 
			T_1 &=  \displaystyle\left(\frac{d}{\tau_1(c,\beta,m,q)}\right)^{\frac{3 - m - \beta - q}{4 - m -\beta}},
		\end{array}
		\right. \quad m + q + \beta < 3,~ c< 0,
	\end{align}
	and finally
	\begin{align}
		\label{gradient_term_3}
		u^q|\nabla u|^m \text{ then solution is }\left\{
		\begin{array}{cll}
			p_1 &= \displaystyle\frac{4 -\beta - m}{3 - m - \beta - q},\\[10pt] 
			\tau_1(c,\beta,m,q) &=  \displaystyle\sqrt[3 - m - \beta - q]{\frac{c\left(3- m - \beta-q\right)^{4 - m - \beta}}{(4 - m - \beta)^{3- m - \beta-q}}},\\[10pt] 
			T_1 &=  \displaystyle\left(\frac{d}{\tau_1(c,\beta,m,q)}\right)^{\frac{3 - m - \beta - q}{4 - m -\beta}},
		\end{array}
		\right.
	\end{align}
	where $m + q + \beta < 3,~c>0$.\\
	If one builds $h(s) = \tau_1s^{p_1}$, then $h$ is a solution of $\Delta_{\infty}^\beta u = c |\nabla u|^{m}$ or $\Delta_{\infty}^\beta u = \pm c u^q|\nabla u|^{m}$. For the full equation \eqref{rad_2_main} one must check that the gradient term is lower order when using this profile. Secondly, we will balance the LHS exponent with the potential exponent. Indeed, the solution of this case is exactly the solution of the first equation \eqref{rad_2_main}, i.e., the function $h(s) = \tau_2 s^{p_2}$ is a solution of $\Delta_{\infty}^\beta u= \lambda \left(|x - x_0| -\rho\right)_{+}^\alpha (u^+)^\gamma$. Again, for the full equation \eqref{rad_2_main} one must check that the potential term is lower order in this regime.\\
	Now, let us put the two cases together. We define the following quantities
	\begin{align}
		\label{cond_2}
		&p := \min\left\{p_1,p_2\right\}, \mbox{ and } \tau = \begin{cases}
			\tau_1(c,\beta,m,q), ~ \mbox{if } p_2 >  p_1\\
			\tau_2(\lambda,\gamma,\beta,\alpha), ~ \mbox{if } p_1 > p_2
		\end{cases},\\\nonumber
		&T = \begin{cases}
			T_1, ~ \mbox{if } p_2 > p_1\\
			T_2, ~ \mbox{if } p_1 > p_2
		\end{cases}, \mbox{ or we can write } T = \left(\frac{d}{\tau}\right)^{1/p}.
	\end{align} 
	Remark that $\tau_1(c,\beta,m,q)$ and $p_1$ are defined in \eqref{gradient_term_1},\eqref{gradient_term_2}, and \eqref{gradient_term_3}, whereas $\tau_2(\lambda,\gamma,\beta,\alpha)$ and $p_2$ are defined in \eqref{radial_sol_c=0}.
	Thus, we can write the general formulation of this case
	\begin{align*}
		h(s) = \tau s^{p}.
	\end{align*}
	\noindent On the other hand, we will consider form (II) of the Hardy-Hénon type equations then \eqref{gene_main_rad} becomes
	\begin{align}
		\label{final_main_1}
		\tag{A2}
		\left\{
		\begin{array}{cccc}
			\Delta_\infty^\beta u &= &cH(u,\nabla u) + \lambda\left[1 + (|x-x_0|-\rho)^2\right]^{-\alpha}u^{\gamma} &\mbox{in } B_R(x_0),\\[5pt]
			u &= &d &\mbox{on }\partial B_R(x_0),
		\end{array}
		\right.
	\end{align}
	where $H = |\nabla u|^m$ or $H = \pm u^q|\nabla u|^m$, $\alpha\in\left[0,\frac{(4-\beta-m)(3-\beta-\gamma)}{2(3-\beta-m)}\right)$, and $d > 0$ . Put $\rho = R - T \in (0,R)$, $s: = |x - x_0|-\rho \in (0,T]$. We consider the following ODE related to \eqref{final_main_1}:
	\begin{align}
		\label{final_ODE}
		(h')^{2-\beta}h'' = cH(h,h') + \lambda\phi(s)h^\gamma,~\mbox{ where }\phi(s): = (1+s^2)^{-\alpha}.
	\end{align}
   Similarly to the general equation above, we cannot balance the exponents with three terms simultaneously. Therefore, we will balance the LHS term against each term on the RHS individually. In fact, depending on the specific form of the Hamiltonian $H$, the exponents $p_1^*$ and $\tau_1^*$ correspond to the respective values calculated in \eqref{gradient_term_1}, \eqref{gradient_term_2}, and \eqref{gradient_term_3}. Therefore, we only need to find the exponents $p_2^*$ and $\tau_2^*$ for the potential term
	 \begin{align}
	 	\label{final_eq: rad_1_1}
	 	h''(s)(h'(s))^{2-\beta} = \lambda\phi(s)h(s)^\gamma \mbox{ in } (0,T].
	 \end{align}
	 satisfying the initial condition: $h(0) = 0$ and $h(T) = d$.
	Repeating the same steps as in form (I), we also get the results when the weight $\phi(s)$ is bounded, which is equivalent to $s$ being finite. Remark that for $0<s\le T$, we can easily check that
	\begin{align*}
		(1+T^2)^{-\alpha} \leq \phi(s) < 1.
	\end{align*}
	To obtain a supersolution of \eqref{final_eq: rad_1_1}, we must opt to analyze the "worst-case" scenario for a more robust argument, i.e., we choose $\phi(s) = (1+T^2)^{-\alpha}$. Hence, we solve the following equation
	\begin{align*}
		h''(s)(h'(s))^{2-\beta} = \lambda (1+T^2)^{-\alpha} h(s)^\gamma \mbox{ in } (0,T).
	\end{align*}
	We deduce
	\begin{align}
		\label{cond:r_finite}
		p_2^* = \displaystyle\frac{4 - \beta}{3 - \beta - \gamma}, \quad \tau_2^*(\lambda,T,\gamma,\beta,\alpha) =(1+T^2)^{-\frac{\alpha}{3-\beta-\gamma}}\tau_2^{**}(\lambda,\gamma,\beta), 
	\end{align}
	where $\tau_2^{**}(\lambda,\gamma,\beta)$ is defined by
	\begin{align}\label{2'}
		\tau_2^{**}(\lambda,\gamma,\beta)=\displaystyle\sqrt[3-\beta-\gamma]{\lambda\cdot\frac{ (3-\beta-\gamma)^{4-\beta}}{(4-\beta)^{3-\beta}(1+\gamma)}}.
	\end{align}
	Consequently, we obtain a supersolution after solving \eqref{final_ODE}
	\begin{align*}
		h(s) =  \chi_{\{p_2^*>p_1^*\}}\tau^*s^{p^*},
	\end{align*}
	where $p^*,\tau^*(\lambda,\gamma,\beta,\alpha)$, and $\chi_{\{p_2^*>p_1^*\}}$ are defined as follows
	\begin{align}\label{cond_3}
		p^*=\min\{p_1^*,p_2^*\}, \quad \tau^*=\begin{cases}
			\tau_1(c,\beta,m,q), \mbox{ if } p_2^* > p_1^*\\
			\tau_2^{**}(\lambda,\gamma,\beta), ~~\mbox{ if } p_1^* > p_2^*
		\end{cases},
	\end{align} 
	and
	\begin{align}
		\label{5'}
		\chi_{\{p_2^*>p_1^*\}} = \left\{\begin{array}{ll}
			\hspace*{1cm} 1,  &\mbox{if } p_2^* > p_1^*\\
			(1+T^2)^{-\frac{\alpha}{3-\beta-\gamma}}  &\mbox{if } p_1^* > p_2^*
		\end{array}\right..
	\end{align}
	Note that $\tau_1(c,\beta,m,q)$ is defined in \eqref{gradient_term_1},\eqref{gradient_term_2}, and \eqref{gradient_term_3} corresponding to three cases of the Hamiltonian, while $\tau_2^{**}(\lambda,\gamma,\beta)$ is defined in \eqref{2'}. Additionally, $p_1^*=p_1$ is given in \eqref{gradient_term_1},\eqref{gradient_term_2}, and \eqref{gradient_term_3}, and $p_2^*$ is given in \eqref{cond:r_finite}. 
	Using the condition $h(T)=d$, we find that $T$ is a solution to the following equation
	\begin{align}\label{1'}
		T=\sqrt[1/p^*]{\dfrac{d}{\chi_{\{p_2^*>p_1^*\}}\tau^*}}.
	\end{align}
	\noindent Fixing $x_0 \in \mathbb{R}^n$ and $0 < \rho < R$, we assume the dead-core compatibility condition
	\begin{align*}
		R > T.
	\end{align*}
	Define the following radially symmetric function $u: B_R(x_0) \setminus B_\rho(x_0) \to \mathbb{R}$ given by
	\begin{align*}
		u(x) = h\left(|x-x_0| - \rho\right),
	\end{align*}
	where $\rho = R -T$. One easily verifies that $u$ pointwise solves the equation
	\begin{align*}
		\Delta_\infty^\beta u(x) = cH(u,\nabla u) + \lambda f(|x - x_0| -\rho,u) \mbox{ in }B_R(x_0) \setminus B_\rho(x_0).
	\end{align*}
	The boundary conditions: $u \equiv 0$ on $\partial B_\rho(x_0)$, and $u \equiv d$ on $\partial B_R(x_0)$ are also satisfied. Moreover, by the construction, for each $y \in \partial B_\rho(x_0)$, we have
	\begin{align*}
		\lim\limits_{x \to y}\nabla u(x) = h'(0^+) \frac{y}{|y|} = 0.
	\end{align*}
	Thus, by extending $u \equiv 0$ in $B_\rho(x_0)$, we obtain a function in $B_R(x_0)$ satisfying 
	\begin{align*}
		\Delta_\infty^\beta u(x) = cH(u,\nabla u) + \lambda f(|x - x_0| -\rho,u) \mbox{ in }B_R(x_0).
	\end{align*}
	We conclude that $u$ is a radial supersolution/subsolution of the full equation provided the other term is lower order 
	\begin{align}
		\label{u_solution_2}
		u(x) =\Lambda \left[|x-x_0| - R + \left(\frac{d}{\Lambda}\right) ^{1/\nu}\right]_+^{\nu}.
	\end{align}
	where $\Lambda,\nu$ can be equal to either $\tau,p$ or $\tau^*,p^*$ depending on the form of the Hardy-Hénon type equation, the reader may see Table \ref{Table_1}. Its plateau is precisely $B_\rho(x_0)$, where
	\begin{align*}
		0 < \rho: = R - \left(\frac{d}{\Lambda}\right) ^{1/\nu}.
	\end{align*}
	Now, if $v$ is an arbitrary solution to
	\begin{align*}
		\Delta_\infty^\beta v = cH(v,\nabla v) + \lambda f(|x - x_0| -\rho,v) \mbox{ in }\Omega \subset \mathbb{R}^n,
	\end{align*}
	and $x_0 \in \Omega$ is an interior point, define $\mathcal{N}: (0,\mathrm{dist}(x_0,\partial \Omega)) \to [0,\infty)$ by
	\begin{align*}
		\mathcal{N}(R) := \sup_{B_R(x_0)}v.
	\end{align*}
	If for some $0 < R < \mathrm{dist}(x_0,\partial \Omega)$, we have
	\begin{align*}
		\mathcal{N}(R) < \Lambda R^{\nu},
	\end{align*}
	Then $x_0$ is a plateau point.\\
	The radial barrier constructed above will be crucial in deriving Liouville theorems in the next section.
	\section{\bf Liouville properties}
	\label{sec: 5}
	In this final section, we present our Liouville-type results for the infinity Laplacian equations introduced above. More precisely, we establish four theorems and their related corollaries. The proofs rely on a subtle combination of the comparison principle and the analysis of the radial version presented in Section \ref{sec: 4}.
	\begin{theorem}[\bf Liouville-type result I]
		\label{theo: Liouville-type result I}
		Given $\rho,\lambda > 0$ and $x_0 \in \mathbb{R}^n$. Let $u$ be a nonnegative entire viscosity solution to 
		\begin{align}
			\label{main_theorem_1_1}
			\Delta_\infty^\beta u = \lambda \mathcal{H}_{\rho,x_0}(|x|)(u^+)^\gamma \mbox{ in } \mathbb{R}^n,
		\end{align} 
		where the weight $\mathcal{H}_{\rho,x_0}(|x|) \lesssim \left(|x - x_0| - \rho\right)^\alpha_+$ and the conditions on the parameters $\gamma,\alpha$ are given in Table \ref{Table_1}. Suppose that
		\begin{align}
			\label{main_theorem_1_2}
			\displaystyle\limsup_{|x|\to\infty} \frac{u(x)}{|x|^{\frac{4-\beta+\alpha}{3-\beta-\gamma}}}  < \sqrt[3-\beta-\gamma]{\lambda\cdot\frac{ (3-\beta-\gamma)^{4-\beta}}{(4-\beta+\alpha)^{3-\beta}(1+\alpha+\gamma)}},
		\end{align}
		then $u \equiv 0$.
	\end{theorem}
	\begin{proof}
		For each $\varepsilon > 0$, let $u_\varepsilon$ be the solution of the following penalized boundary value problem
		\begin{align*}
			\left\{
			\begin{array}{cccc}
				\Delta_\infty^\beta u_\varepsilon &= &\lambda \left(|x - x_0| - \rho\right)^\alpha_+(u_\varepsilon^+)^\gamma + \varepsilon &\mbox{in } B_R(x_0),\\[5pt]
				u_\varepsilon &= &\displaystyle\sup_{\partial B_R(x_0)} u&\mbox{on } \partial B_R(x_0).
			\end{array}
			\right.
		\end{align*}
		Now, fixing $R > \rho > 0$, let us consider $v: \overline{B_R(x_0)} \to \mathbb{R}$, the radial solution to the boundary value problem
		\begin{align*}
			\left\{
			\begin{array}{cccc}
				\Delta_\infty^\beta v &= &\lambda \left(|x - x_0| - \rho\right)^\alpha_+(v^+)^\gamma &\mbox{in } B_R(x_0),\\[5pt]
				v &= &\displaystyle\sup_{\partial B_R(x_0)} u &\mbox{on } \partial B_R(x_0).
			\end{array}
			\right.
		\end{align*}
		Thanks to the assumption $\mathcal{H}_{\rho,x_0}(|x|) \lesssim \left(|x - x_0| - \rho\right)^\alpha_+$, we use the comparison principle (Theorem \ref{Comp.Prin}) to conclude that 
		\begin{align}
			\label{proof_1_1}
			u \leq u_\varepsilon \mbox{ in } B_R(x_0).
		\end{align}
		By using the stability (Lemma \ref{lem_Stab}) and the uniqueness of the solution to the Dirichlet problem (Theorem \ref{theo:sandwich}), we ensure that there is a subsequence $(u_{\varepsilon_j})$ such that $u_{\varepsilon_j} \to v$ uniformly on $\overline{B_R(x_0)}$. Hence, from \eqref{proof_1_1}, we obtain 
		\begin{align}
			\label{proof_1_*}
			u \leq v \mbox{ in } B_R(x_0).
		\end{align}
		It follows from hypothesis \eqref{main_theorem_1_2}, taking $R$ sufficiently large ($R\gg 1$) such that
		\begin{align}
			\label{proof_1_2}
			\displaystyle\sup_{\partial B_R(x_0)} \frac{u(x)}{R^{\frac{4-\beta+\alpha}{3-\beta-\gamma}}} \leq \Phi \theta,\text{ with }\theta=\sqrt[3-\beta-\gamma]{\lambda\cdot\frac{ (3-\beta-\gamma)^{4-\beta}}{(4-\beta+\alpha)^{3-\beta}(1+\alpha+\gamma)}},
		\end{align}
		for some $0<\Phi <1$. From \eqref{u_solution_2}, it follows that the solution $v = v_R$ is given by
		\begin{align}
			\label{proof_1_3}
			v(x) =\theta \left[|x-x_0| - R + \left(\frac{\displaystyle\sup_{\partial B_R(x_0)} u}{\theta}\right)^{\frac{3 - \beta - \gamma}{4 - \beta + \alpha}}\right]^{\frac{4-\beta + \alpha}{3 - \beta - \gamma}}_+.
		\end{align}
		Combining \eqref{proof_1_2} and \eqref{proof_1_3}, we deduce
		\begin{align}\label{proof_1'_1}
			v(x) &\leq \theta \left[|x-x_0| - R\left(1 - \Phi^{\frac{3 - \beta - \gamma}{4 - \beta + \alpha}}\right)\right]^{\frac{4-\beta + \alpha}{3 - \beta - \gamma}}_+.
		\end{align}
		Combining \eqref{proof_1_*} and \eqref{proof_1'_1}, we get
		\begin{align*}
			u(x) \leq \theta \left[|x-x_0| - R\left(1 - \Phi^{\frac{3 - \beta - \gamma}{4 - \beta + \alpha}}\right)\right]^{\frac{4-\beta + \alpha}{3 - \beta - \gamma}}_+.
		\end{align*}
		Letting $R \to \infty$, we conclude the proof of the theorem.
	\end{proof}
	\begin{corollary}
		The statement of Theorem \ref{theo: Liouville-type result I} is sharp, in the sense that we cannot extend the strict inequality in \eqref{main_theorem_1_2}. Indeed, take the function $h:\mathbb R^n\to\mathbb R$ and the weight $\mathcal{H}_{\rho,x_0}:\mathbb R^n\to \mathbb R$ defined as
		\begin{align*}
			h(x) = \sqrt[3-\beta-\gamma]{\lambda\cdot\frac{ (3-\beta-\gamma)^{4-\beta}}{(4-\beta+\alpha)^{3-\beta}(1+\alpha+\gamma)}}|x|^{{\frac{4-\beta + \alpha}{3 - \beta - \gamma}}}\text{ and }\mathcal{H}_{\rho,x_0}(|x|)=|x|^{\alpha}.
		\end{align*}
		A straightforward calculation indicates that $h$ is a solution to \eqref{main_theorem_1_1}. Furthermore, we can easily check that
		\begin{align*}
			\limsup_{|x|\to\infty}{\dfrac{h(x)}{|x|^{\frac{4-\beta+\alpha}{3-\beta-\gamma}}}}=\sqrt[3-\beta-\gamma]{\lambda\cdot\frac{ (3-\beta-\gamma)^{4-\beta}}{(4-\beta+\alpha)^{3-\beta}(1+\alpha+\gamma)}}.
		\end{align*}
	\end{corollary}
	\begin{corollary}
		The conclusion of Theorem \ref{theo: Liouville-type result I} is still valid if $u$ is a subsolution to problem \eqref{main_theorem_1_1}, and the proof in this case is entirely similar to that of Theorem \ref{theo: Liouville-type result I}.
		\end{corollary}
		\begin{corollary}
			The result of Theorem \ref{theo: Liouville-type result I} is violated if $u$ is a supersolution to problem \eqref{main_theorem_1_1}. For example, $u\equiv c$, where $c$ is a positive constant, is a supersolution to \eqref{main_theorem_1_1} since
			\begin{align*}
				\Delta_{\infty}^{\beta}u=0\le \lambda \mathcal{H}_{\rho,x_0}(|x|) (u^+)^{\gamma} \text{ in }\mathbb R^n.
			\end{align*}
			However, we have
			\begin{align*}
		\limsup_{|x|\to\infty}{\dfrac{u(x)}{|x|^{\frac{4-\beta+\alpha}{3-\beta-\gamma}}}}=0<\sqrt[3-\beta-\gamma]{\lambda\cdot\frac{ (3-\beta-\gamma)^{4-\beta}}{(4-\beta+\alpha)^{3-\beta}(1+\alpha+\gamma)}}.
			\end{align*}
		\end{corollary}
	\begin{theorem}[\bf Liouville-type result II]
		\label{theo: Liouville-type result II}
		Given $\rho,\lambda> 0$ and $x_0 \in \mathbb{R}^n$, let $u$ be a nonnegative entire viscosity solution to 
		\begin{align}
			\label{main_theorem_2_1}
			\Delta_\infty^\beta u = cH(u,\nabla u) + \lambda \mathcal{H}_{\rho,x_0}(|x|)(u^+)^\gamma \mbox{ in } \mathbb{R}^n,
		\end{align} 
		where the weight $\mathcal{H}_{\rho,x_0}(|x|) \lesssim \left(|x - x_0| - \rho\right)^\alpha_+$, and the conditions on the parameters $m,\gamma,\alpha,q$ are given in Table \ref{Table_1}. Suppose that
		\begin{align}
			\label{main_theorem_2_2}
			\displaystyle\limsup_{|x|\to\infty} \frac{u(x)}{|x|^{p}}  < \tau,
		\end{align}
		where $p$ and $\tau$ are defined in \eqref{cond_2}. Then $u \equiv 0$.
	\end{theorem}
	\begin{proof}
		We consider the boundary value problem of \eqref{rad_2_main} after adding the solution to the penalized problem $u_\varepsilon$
		\begin{align}\label{problem2}
			\left\{
			\begin{array}{cccc}
				\Delta_\infty^\beta u_\varepsilon &= &cH(u_{\varepsilon},\nabla u_{\varepsilon}) + \lambda \left(|x - x_0| - \rho\right)^\alpha_+(u_\varepsilon^+)^\gamma + \varepsilon &\mbox{in } B_R(x_0),\\[5pt]
				u_\varepsilon &= &\displaystyle\sup_{\partial B_R(x_0)} u&\mbox{on } \partial B_R(x_0),
			\end{array}
			\right.
		\end{align}
		for each $\varepsilon > 0$. Next, let us consider $v: \overline{B_R(x_0)} \to \mathbb{R}$ and fix $R > \rho >0$, we consider the radial solution to the boundary value problem
		\begin{align*}
			\left\{
			\begin{array}{cccc}
				\Delta_\infty^\beta v &= &cH(v,\nabla v) + \lambda \left(|x - x_0| - \rho\right)^\alpha_+(v^+)^\gamma &\mbox{in } B_R(x_0),\\[5pt]
				v &= &\displaystyle\sup_{\partial B_R(x_0)} u &\mbox{on } \partial B_R(x_0).
			\end{array}
			\right.
		\end{align*}
		An analogous argument to the proof of Theorem \ref{theo: Liouville-type result I}, using key tools such as the comparison principle (Theorem \ref{Comp.Prin}), the uniqueness (Theorem \ref{theo:sandwich}), and the stability (Lemma \ref{lem_Stab}), also yields
		\begin{align}
			\label{proof_2_*}
			u \leq v \mbox{ in } B_R(x_0).
		\end{align}
		By hypothesis \eqref{main_theorem_2_2}, choose $R$ sufficiently large ($R\gg 1$) such that
		\begin{align}
			\label{proof_2_1}
			\displaystyle\sup_{\partial B_R(x_0)} \frac{u(x)}{R^{p}} \leq \Phi \tau,
		\end{align}
		for some $0<\Phi <1$. Thanks to \eqref{u_solution_2}, we obtain the function $v$ defined as
		\begin{align}
			\label{proof_2_2}
			v(x) =\tau \left[|x-x_0| - R + \displaystyle\left(\frac{\displaystyle\sup_{\partial B_R(x_0)} u}{\tau}\right)^{\frac{1}{p}}\right]^{p}_+,
		\end{align}
		is a supersolution to problem \eqref{problem2}.
		Combining the estimates \eqref{proof_2_*},\eqref{proof_2_1} and \eqref{proof_2_2}, we arrive at
		\begin{align}
			\label{proof_2_3}
			u(x) \leq \tau \left[|x-x_0| - R\left(1 - \Phi^{\frac{1}{p}}\right)\right]^{p}_+.
		\end{align}
		Letting $R \to \infty$ in \eqref{proof_2_3}, we conclude $u(x) \leq 0$. Note that by assumption, $u(x) \geq 0$, then $u\equiv 0$.
	\end{proof}
	\begin{corollary}
		We cannot replace the assumption to $u$ is a supersolution to problem \eqref{main_theorem_2_1}. A counterexample is the function $u(x)=c$ for $x\in\mathbb R^n$ and $c$ is a positive constant.
	\end{corollary}
	\begin{corollary}
		The conclusion of Theorem \ref{theo: Liouville-type result II} still holds if $u$ is a subsolution to the problem \eqref{main_theorem_2_1}. Moreover, we can apply similar arguments as the proof of Theorem \ref{theo: Liouville-type result II} to prove this case.
	\end{corollary}
	\begin{theorem}[\bf Liouville-type result III]
		\label{theo: Liouville-type result - III}
		Let $\rho,\lambda > 0$ and $\alpha\in\left[0,\frac{(4-\beta-m)(3-\beta-\gamma)}{2(3-\beta-m)}\right)$. Assume that $u$ is a nonnegative entire viscosity solution to
		\begin{equation}\label{E}
			\Delta_{\infty}^{\beta} u
			= cH(u,\nabla u) + \lambda\phi_{\rho,x_0}(|x|)\,u^{\gamma}, 
			\qquad x\in\R^n.
		\end{equation}
		where the weight $\phi_{\rho,x_0}(|x|) \lesssim [1+(|x-x_0|-\rho)^2]^{-\alpha}$ for some $x_0\in\mathbb R^n$, and the conditions on the parameters $m, q, \gamma$ are given in Table \ref{Table_1}. If the following condition holds
		\begin{align}
			\label{final_cond_1}
			\displaystyle\limsup_{|x|\to\infty} \frac{u(x)}{|x|^{p^*}(1+|x|^2)^{-\frac{\alpha}{3-\beta-\gamma}}}  < \tau^{*},
		\end{align}
		where $p^*,\tau^{*}$ are respectively defined in \eqref{cond_3}, then every nonnegative viscosity solution of \eqref{E} is trivial, that is $u\equiv0$ in $\mathbb{R}^n$.
	\end{theorem}
	\begin{proof}
		Fixing $R > \rho >0$, let $v: \overline{B_R(x_0) }\to \mathbb{R}$, we consider the radial solution to the boundary value problem
		\begin{align*}
			\left\{
			\begin{array}{cccc}
				\Delta_\infty^\beta v &=& cH(v,\nabla v) + \lambda\left[1 + (|x-x_0|-\rho)^2\right]^{-\alpha}v^{\gamma} &\mbox{in } B_R(x_0),\\[5pt]
				v &= &\displaystyle\sup_{\partial B_R(x_0)} u &\mbox{on } \partial B_R(x_0).
			\end{array}
			\right.
		\end{align*}
		By comparison principle, Theorem \ref{Comp.Prin}, $u \leq v$ in $B_R(x_0)$ (note that the same proof works since the function $f(x,t)=(1+x^2)^{-\alpha}t^{\gamma}$ is strictly increasing in the second variable). It follows by hypothesis \eqref{final_cond_1} that, taking $R$ sufficiently large ($R\gg 1$) such that
		\begin{align}\label{final_proof_1_2}
			\displaystyle\sup_{\partial B_R(x_0)} \frac{u(x)}{R^{p^*}(1+R^2)^{-\frac{\alpha}{3-\beta-\gamma}}} &\leq \Phi\tau^{*},
		\end{align}
		for some $0<\Phi <1$. Thanks to \eqref{u_solution_2}, we obtain the supersolution $v$ as follows
		\begin{align}
			\label{final_proof_1_3}
			v(x) =\chi_{\{p_2^*>p_1^*\}} \tau^{*} \left[|x-x_0| - R + \left(\frac{\displaystyle\sup_{\partial B_R(x_0)} u}{\chi_{\{p_2^*>p_1^*\}}\tau^{*}}\right)^{1/p^*}\right]^{p^*}_+,
		\end{align}
		where $\chi_{\{p_2^*>p_1^*\}}$ is defined in \eqref{5'}. Finally, combining \eqref{final_proof_1_2} and \eqref{final_proof_1_3}, we deduce 
		\begin{align}
			\label{final_proof_2_3}
			u(x) \leq \tau^{*} \left[|x-x_0| - R\left(1 - \Phi^{1/p^*}\right)\right]^{p^*}_+.
		\end{align}
		Letting $R \to \infty$ in \eqref{final_proof_2_3}, we conclude the proof of Theorem \ref{theo: Liouville-type result - III}.
	\end{proof}
\begin{remark}
	Let us explain the domain of the parameter $\alpha$ in Theorem \ref{theo: Liouville-type result - III}. The mentioned condition of $\alpha$ ensures that the function $h(x)=|x|^{p^*}(1+|x|^2)^{-\frac{\alpha}{3-\beta-\gamma}},x\in\mathbb R^n$ converges to $+\infty$ when $|x|\to+\infty$. It allows the growth condition \eqref{final_cond_1} to cover a broad class of viscosity solutions, including unbounded solutions.
\end{remark}
Finally, we establish a Liouville result for nonlinearity $|u|^{\gamma}e^u$. Our proof of this theorem is based on the doubling method from \cite{AAE2025}, along with a suitable adjustment of the coupling function. Furthermore, in order to state our result in a general sense, we allow for a functional coefficient $a(x)$ for the gradient term. As a result, our theorem deals with a more general framework compared to \cite[Theorem 2.6]{AH2020}.
\begin{theorem}[\bf Liouville-type result IV]
	\label{cxmodel}
	Suppose that $u\in C(\mathbb R^n)$ is a viscosity solution to
	\begin{align*}
		\Delta_{\infty}^{\beta}u=a(x)(|x|+1)^{\alpha}|\nabla u|^m+\lambda|u|^{\gamma}e^u\mbox{ in }\mathbb R^n,
	\end{align*}
	where $m\in (0,2-\beta]$, $a:\mathbb R^n\to (0,\infty)$ is a function satisfying $\inf_{\mathbb R^n}a>0$, $\lambda,\gamma\geq 0$, and $\alpha\in (-1,0]$.
	If the following condition holds
	\begin{align}\label{e}
		\lim_{R\to\infty}{\dfrac{\mathrm{osc}_{B_R}u}{R}}=0,
	\end{align}
	where $\mathrm{osc}_{B_R}u=\sup_{B_R}u-\inf_{B_R}u$, for each $R>0$, then $u$ is necessarily constant.
\end{theorem}
\begin{proof}
	Let $L_R=\frac{\mathrm{osc}_{B_R}u}{R}$ and $\kappa_R=\frac{L_R}{R}$, for every $R>0$. We consider the function
	\begin{align*}
		\Phi_R(x,y)=u(x)-u(y)-\kappa_R\omega_R(|x-y|)-2\kappa_R(|x|^2+|y|^2),
	\end{align*}
	where $\omega_R(t)=-t+\vartheta t^2$. Let us choose $\vartheta=\vartheta(R)>0$ small enough so that
	\begin{align*}
		-1\le \omega_R'(t)\le -\dfrac{1}{2}\text{ and } \omega_R''(t)\geq 0,\text{ for }t\in [0,R].
	\end{align*}
	Since $\Phi_{R}$ is continuous in $\overline{B_R}\times\overline{B_R}$, $\Phi_R$ attains its maximum value on this domain. We consider two separate cases in this proof.\medskip
	
	\textbf{Case 1}: There exists $R_0>0$ such that $\max\limits_{\overline{B_R}\times\overline{B_R}}{\Phi_R}\le 0$, for every $R>R_0$.
	
	Fixing two points $x,y\in\mathbb R^n$, we are able to find $R_1>R_0$ such that $x,y\in B_{R/2}$, for $R>R_1$. Thus, we derive the following estimates
	\begin{align}\label{1}
		|u(x)-u(y)|&\le \kappa_R\omega_R(|x-y|)+2\kappa_R(|x|^2+|y|^2)\notag\\
		&\le \kappa_R|x-y|^2+2\kappa_R(|x|^2+|y|^2), \text{ for all }x,y\in B_{\frac{R}{2}}\text{ and }R>R_1.  
	\end{align}
	Let $R\to \infty$ in \eqref{1}, noting that $\kappa_R\to 0$ by assumption \eqref{e}, we find that $u(x)=u(y)$ for all $x,y\in B_{R/2}$ and $R>R_1$. Since $x,y$ are taken arbitrarily, we conclude $u$ is constant.\medskip

	\textbf{Case 2}: There exists a sequence $R_i\to\infty$ such that $M_i=\max\limits_{\overline{B_{R_i}}\times\overline{B_{R_i}}}{\Phi_{R_i}}> 0$, for every $i\in\mathbb N$. 
	
	We remark that $L_{R_i}>0$, for every $i\in \mathbb N$. Indeed, if there is a number $j\in\mathbb N$ so that $L_{R_j}=0$, that means $\mathrm{osc}_{B_{R_j}}u=0$ and $\kappa_{R_j}=0$. Combining it with the estimate
	\begin{align*}
		\Phi_{R_j}(x,y)=u(x)-u(y)-\kappa_{R_j}\omega_{R_j}(|x-y|)-2\kappa_{R_j}(|x|^2+|y|^2)= u(x)-u(y)\le \mathrm{osc}_{B_{R_j}}{u},~ x,y\in \overline{B_{R_j}},
	\end{align*}
	we deduce that $M_j\le 0$, which is a contradiction. 
	
	For simplicity, throughout the rest of this proof, let us use the notation 
	\begin{align*}
		L_i:=L_{R_i}, \kappa_i:=\kappa_{R_i}, \Phi_i:=\Phi_{R_i},\omega_{R_i}:=\omega_i,~i\in\mathbb N.
	\end{align*}
	Assume that the function $\Phi_i$ attains its maximum value on $\overline{B_{R_i}}\times \overline{B_{R_i}}$ at $(x_i,y_i)$. Since $\Phi_i(x_i,y_i)>0$ then
	\begin{align}\label{3.1}
		\kappa_i\omega_i(|x_i-y_i|)+2\kappa_i(|x_i|^2+|y_i|^2)<u(x_i)-u(y_i).
	\end{align}
	It follows from \eqref{3.1} that
	\begin{align*}
		2\kappa_i|x_i|^2<u(x_i)-u(y_i)+\kappa_i|x_i-y_i|\le 2\mathrm{osc}_{B_{R_i}}u,
	\end{align*}
	which leads to $|x_i|<R_i$. With similar arguments, we also obtain $|y_i|<R_i$. This reveals that $(x_i,y_i)$ is a local maximum of $\Phi_i$ over $\overline{B_{R_i}}\times\overline{B_{R_i}}$.  Additionally, thanks to \eqref{3.1}, we obtain
	\begin{align*}
		u(x_i)-u(y_i)>2\kappa_i(|x_i|^2+|y_i|^2)-\kappa_i|x_i-y_i|,
	\end{align*}
	thus we must have $x_i\neq y_i$.
	
	Then applying \cite[Theorem 3.2]{CIL1992}, there exists $X_i,Y_i\in \mathbb S^{n\times n}$ such that $(p_i,X_i)\in \overline{J}^{2,+}_{B_{R_i}}(u(x_i)),\\(q_i,Y_i)\in \overline{J}^{2,-}_{B_{R_i}}(u(y_i))$ and
	\begin{align}\label{3}
		\left(\begin{matrix}
			X_i&0\\
			0&-Y_i
		\end{matrix}\right)\le A+\varepsilon A^2,\text{ for each }\varepsilon>0,
	\end{align}
	where $A=D^2\Theta_i(x_i,y_i)$ and $\Theta_i(x,y)=\kappa_i\omega_i(|x-y|)+2\kappa_i(|x|^2+|y|^2)$. A direct calculation shows that
	\begin{align}\label{4}
		A=\left(\begin{matrix}
			Z&-Z\\
			-Z& Z
		\end{matrix}\right)+4\kappa_iI_{2n\times 2n},
	\end{align}
	and
	\begin{align*}
		Z=\kappa_i\omega_i''(|x_i-y_i|)\dfrac{(x_i-y_i)(x_i-y_i)^T}{|x_i-y_i|^2}+\kappa_i\dfrac{\omega_i'(|x_i-y_i|)}{|x_i-y_i|}\left(I-\dfrac{(x_i-y_i)(x_i-y_i)^T}{|x_i-y_i|^2}\right).
	\end{align*}
	From \eqref{3} and \eqref{4}, we arrive at
	\begin{align*}
		\left(\begin{matrix}
			p\\
			q
		\end{matrix}\right)^T\left(\begin{matrix}
			X_i&0\\
			0&-Y_i
		\end{matrix}\right)\left(\begin{matrix}
			p\\
			q
		\end{matrix}\right)&\le \left(\begin{matrix}
			p\\
			q
		\end{matrix}\right)^TA
		\left(\begin{matrix}
			p\\
			q
		\end{matrix}\right)+\varepsilon\left(\begin{matrix}
			p\\
			q
		\end{matrix}\right)^TA^2
		\left(\begin{matrix}
			p\\
			q
		\end{matrix}\right).
	\end{align*}
	Choosing $(p,q)=(p_i,0)$ and $\varepsilon>0$ small enough, we deduce that
	\begin{align}\label{5}
		\langle X_ip_i,p_i\rangle \le C(\langle Zp_i,p_i\rangle+\langle \kappa_ip_i,p_i\rangle),\text{ for some }C>0.
	\end{align}
	Moreover, using $\omega'<0$ and $\omega''\geq 0$, we derive the following estimates
	\begin{align}\label{6}
		\langle Zp_i,p_i\rangle &=\kappa_i\omega_i''(|x_i-y_i|)\left\langle\dfrac{(x_i-y_i)(x_i-y_i)^T}{|x_i-y_i|^2}p_i,p_i \right\rangle\notag\\
		&\quad +\kappa_i\dfrac{\omega_i'(|x_i-y_i|)}{|x_i-y_i|}\left\langle \left(I-\dfrac{(x_i-y_i)(x_i-y_i)^T}{|x_i-y_i|^2}\right)p_i,p_i\right\rangle\notag\\
		&\le \kappa_i\omega_i''(|x_i-y_i|)|p_i|^2\notag\\
		&\le C\kappa_i|p_i|^2.
	\end{align}
	From \eqref{5} and \eqref{6}, we obtain 
	\begin{align}\label{6.1}
		\langle X_ip_i,p_i\rangle \le C\kappa_i|p_i|^2.
	\end{align}
	Applying the definition of viscosity subsolution, we obtain
	\begin{align}\label{6.2}
		|p_i|^{-\beta}\langle X_ip_i,p_i\rangle\geq a(x_i)(|x_i|+1)^{\alpha}|p_i|^m.
	\end{align}
Thanks to the choice of the function $\omega_i$, we arrive at
	\begin{align*}
		|p_i|=\left|\kappa_i\omega_i'(|x_i-y_i|)\dfrac{x_i-y_i}{|x_i-y_i|}+4\kappa_ix_i\right|&\le \kappa_i|\omega_i'(|x_i-y_i|)|+4\kappa_i|x_i|\\
		&\le \kappa_i+4\kappa_iR_i\\
		&\le 5L_i.
	\end{align*}
	Combining estimates \eqref{6.1} and \eqref{6.2} gives
	\begin{align}\label{9}
		a(x_i)(|x_i|+1)^{\alpha}&\le C\kappa_i|p_i|^{2-\beta-m}.
	\end{align}
	Note that $\kappa_i=\dfrac{L_i}{R_i}$, we derive from \eqref{9} that
	\begin{align}\label{10}
		a(x_i)(|x_i|+1)^{\alpha}&\le C\dfrac{L_i^{3-\beta-m}}{R_i}.
	\end{align}
	From \eqref{10}, it follows that
	\begin{align*}
		\inf\limits_{\mathbb R^n}{a}\le a(x_i)&\le C\dfrac{L_i^{3-\beta-m}}{R_i}(|x_i|+1)^{-\alpha}\\
		&\le C\dfrac{L_i^{3-\beta-m}}{R_i}(R_i+1)^{-\alpha}\\
		&\le C {L_i^{3-\beta-m}}\dfrac{(R_i+1)^{-\alpha}}{R_i^{-\alpha}} R_i^{-\alpha-1}.
	\end{align*}
	Let $i\to\infty$, we obtain $\inf\limits_{\mathbb R^n}{a}\le 0$. This is a contradiction, thus the \textbf{Case 2} cannot hold. The proof is completed.
\end{proof}
\begin{corollary}
	The conclusion of Theorem \ref{cxmodel} is violated if $u$ is a supersolution. For example, take \\$u(x)=1-e^{-|x|^2}$, $a(x)=a$, and $m\in (0,2-\beta]$ where $a$ is a positive constant chosen later. A direct calculation reveals that
	\begin{align*}
		&\Delta_{\infty}^{\beta}u=2^{3-\beta}|x|^{2-\beta}e^{(\beta-3)|x|^2}(1-2|x|^2),\\
		&a(x)(|x|+1)^{\alpha}|\nabla u|^m=a(x)(|x|+1)^{\alpha}2^m|x|^me^{-m|x|^2}.
	\end{align*}
	For $x=0$, we easily arrive at 
	\begin{align}\label{14*}
		\Delta_{\infty}^{\beta}u(0)\le a(0)(|0|+1)^{\alpha}|\nabla u(0)|^m+\lambda |u(0)|^{\gamma}e^{u(0)}.
	\end{align}
	For $|x|\geq \dfrac{\sqrt{2}}{2}$, we obtain $\Delta_{\infty}^{\beta}u\le 0$, thus
	\begin{align}\label{14}
		\Delta_{\infty}^{\beta}u\le a(x)(|x|+1)^{\alpha}|\nabla u|^m+\lambda|u|^{\gamma}e^u,\text{ for }|x|\geq\dfrac{\sqrt{2}}{2}.
	\end{align}
	For $|x|\le \dfrac{\sqrt{2}}{2},x\neq 0$, we derive the following estimates
	\begin{align*}
		\dfrac{\Delta_{\infty}^{\beta}u}{a(x)(|x|+1)^{\alpha}+\lambda|u|^{\gamma}e^u}\le \dfrac{\Delta_{\infty}^{\beta}u}{a(x)(|x|+1)^{\alpha}|\nabla u|^m}&=\dfrac{2^{3-\beta-m}}{a}(|x|+1)^{-\alpha}|x|^{2-\beta-m}e^{(m+\beta-3)|x|^2}(1-2|x|^2)\\
		&\le \dfrac{2^{3-\beta-m}}{a}2^{-\alpha}=\dfrac{2^{3-\beta-m-\alpha}}{a}.
	\end{align*}
	Choose $a=2^{3-\beta-m-\alpha}$, we readily obtain
	\begin{align}\label{15}
		\Delta_{\infty}^{\beta}u\le a(x)(|x|+1)^{\alpha}|\nabla u|^m+\lambda|u|^{\gamma}e^u,\text{ for }|x|\le \dfrac{\sqrt{2}}{2} \text{ and }x\neq 0.
	\end{align}
	From \eqref{14*}, \eqref{14}, and \eqref{15}, we conclude
	\begin{align*}
		\Delta_{\infty}^{\beta}u\le a(x)(|x|+1)^{\alpha}|\nabla u|^m+\lambda|u|^{\gamma}e^u,\text{ for }x\in\mathbb R^n.
	\end{align*}
	However, a straightforward calculation demonstrates that
	\begin{align*}
		\lim_{R\to\infty}{\dfrac{\mathrm{osc}_{B_R}u}{R}}=\lim_{R\to\infty}{\dfrac{1-e^{-R^2}}{R}}=0.
	\end{align*}
\end{corollary}
\begin{corollary}
	The result of Theorem \ref{cxmodel} is still valid if $u$ is a subsolution with the similar proof as above.
\end{corollary}

\textbf{Acknowledgment.}
The authors gratefully acknowledge the support of the Vietnam Institute for Advanced Study in Mathematics (VIASM)
through the program \textit{Summer School ``Research Experience for Undergraduates'' 2025}.

\end{document}